\documentclass[a4paper]{amsart}
\usepackage{latexsym,bm,stmaryrd}
\usepackage{amsmath,amsthm,amsfonts,amssymb,mathrsfs,pb-diagram}
\usepackage[a4paper,hmargin=26mm,vmargin=28mm]{geometry}
\usepackage{microtype}
\usepackage{mathtools}
\usepackage{mathbbol,wasysym}

\newcommand\blam{{\boldsymbol\lambda}}
\newcommand\brho{{\boldsymbol\rho}}
\newcommand\bmu{{\boldsymbol\mu}}
\newcommand\bnu{{\boldsymbol\nu}}

\let\<=\langle
\let\>=\rangle
\def\({\big(}
\def\){\big)}
\def\Z{\mathbb{Z}}

\def\N{\mathbb{N}}

\def\t{\mathfrak{t}}

\def\lam{\lambda}
\def\Lam{\Lambda}
\def\Sym{\mathfrak{S}}

\newcommand\RR{\mathscr{R}}
\newcommand\R[1][n]{\RR_{#1}^{\Lambda}}

\newcommand\HH{\mathscr{H}}

\DeclareMathOperator\Std{Tab}

\def\P{\mathscr{P}}

\def\wf{\widetilde{F}}

\newcommand\bi{\mathbf{i}}

\newcommand\maxim{\text{max}}
\newcommand\minum{\text{min}}

\DeclareMathOperator\End{End}

\DeclareMathOperator\rad{rad}
\DeclareMathOperator\cha{char}
\DeclareMathOperator\tr{tr}

\DeclareMathOperator\Tr{Tr}

\title[cyclotomic nilHecke algebras]
{On the structure of cyclotomic nilHecke algebras}
\author{Jun Hu}\address{School of Mathematical and Statistics\\
  Beijing Institute of Technology\\
  Beijing, 100081, P.R. China}
\email{junhu404@bit.edu.cn}

\author{Xinfeng Liang}\address{School of Mathematical and Statistics\\
  Beijing Institute of Technology\\
  Beijing, 100081, P.R. China}
\email{lxfrd@163.com}

\numberwithin{equation}{section}
\newtheorem{prop}[equation]{Proposition}
\newtheorem{thm}[equation]{Theorem}
\newtheorem{cor}[equation]{Corollary}
\newtheorem{conj}[equation]{Conjecture}

\newtheorem{lem}[equation]{Lemma}

\newtheorem{ques}[equation]{Question}
\theoremstyle{definition}
\newtheorem{dfn}[equation]{Definition}
\theoremstyle{remark}
\newtheorem{rem}[equation]{Remark}
\begin{document}

\bibliographystyle{andrew}

\begin{abstract} In this paper we study the structure of the cyclotomic nilHecke algebras $\HH_{\ell,n}^{(0)}$, where $\ell,n\in\N$. We construct a monomial basis for $\HH_{\ell,n}^{(0)}$ which verifies a conjecture of Mathas. We show that the graded basic algebra of $\HH_{\ell,n}^{(0)}$ is commutative and hence isomorphic to the center $Z$ of $\HH_{\ell,n}^{(0)}$. We further prove that $\HH_{\ell,n}^{(0)}$ is isomorphic to the full matrix algebra over $Z$ and construct an explicit basis for the center $Z$. We also construct a complete set of pairwise orthogonal primitive idempotents of $\HH_{\ell,n}^{(0)}$. Finally, we present a new homogeneous symmetrizing form $\Tr$ on $\HH_{\ell,n}^{(0)}$ by explicitly specifying its values on a given homogeneous basis of $\HH_{\ell,n}^{(0)}$ and show that it coincides with Shan--Varagnolo--Vasserot's symmetrizing form $\Tr^{\text{SVV}}$ on $\HH_{\ell,n}^{(0)}$.
\end{abstract}

\maketitle
\setcounter{tocdepth}{1}

\section{Introduction}

Quiver Hecke algebras $\RR_\alpha$ and their finite dimensional quotients $\R[\alpha]$ (i.e., cyclotomic quiver Hecke algebras) have been hot topics in recent years. These algebra are remarkable because they can be used to categorify quantum groups and their integrable highest weight modules, see \cite{KK}, \cite{KhovLaud:diagI}, \cite{Rou0}, \cite{Rou1} and \cite{VaragnoloVasserot:CatAffineKLR}. These algebras can be regarded as some $\Z$-graded analogues of the affine Hecke algebras and their finite dimensional quotients. Many results concerning the representation theory of the affine Hecke algebras and the cyclotomic Hecke algebras of type $A$  have their $\Z$-graded analogues for the quiver Hecke algebras $\RR_\alpha$ and the cyclotomic quotients $\R[\alpha]$, see \cite{BK:GradedDecomp}, \cite{BKW:GradedSpecht} and \cite{LV}. It is natural to expect that the structure of the affine Hecke algebras and the cyclotomic Hecke algebras of type $A$  also have their $\Z$-graded analogues for the algebras $\RR_\alpha$ and $\R[\alpha]$. In fact, this is indeed the case for the quiver Hecke algebras $\RR_\alpha$. For example, we have faithful polynomial representations, standard basis and a nice description of the center for the algebra $\RR_\alpha$ in a similar way as in the case of the affine Hecke algebras of type $A$. However, the situation turns out to be much more tricky for the cyclotomic quiver Hecke algebras $\R[\alpha]$. Only partial progress have been made for the structure of the cyclotomic quiver Hecke algebras $\R[\alpha]$ so far. For example, \begin{enumerate}
\item The cyclotomic quiver Hecke algebra of type $A$ has a $\Z$-graded cellular basis by \cite{HuMathas:GradedCellular};
\item The cyclotomic quiver Hecke algebra is a $\Z$-graded symmetric algebra by \cite{SVV};
\item The center of the cyclotomic quiver Hecke algebra $\R[\alpha]$ is the image of the center of the quiver Hecke algebra $\RR_\alpha$ whenever the associated Cartan matrix is symmetric of finite type by \cite{Web}.
\end{enumerate}

Note that apart from the type $A$ case, one does not even know any explicit bases for arbitrary cyclotomic quiver Hecke algebras. On the other hand, for the classical cyclotomic Hecke algebra of type $A$, we have not only a Dipper-James-Mathas's cellular basis \cite{DJM:cyc} but also a monomial basis (or Ariki--Koike basis) by \cite{AK1}. But even for the cyclotomic quiver Hecke algebra of type $A$ we do not know any explicit monomial basis. This motivates our first question:

\begin{ques} \label{QA} Can we construct an explicit monomial basis for any cyclotomic quiver Hecke algebra?
\end{ques}

In \cite{SVV}, Shan, Varagnolo and Vasserot have shown that each cyclotomic quiver Hecke algebra can be endowed with a homogeneous symmetrizing form $\Tr^{\text{SVV}}$ which makes it into a graded symmetric algebra (see Remark \ref{FinalRem} and \cite[Section 6.3]{HuMathas:GradedCellular} for the type $A$ case). However,  the SVV symmetrizing form $\Tr^{\text{SVV}}$ is defined in an inductive manner. It is difficult to compute the explicit value of the form $\Tr^{\text{SVV}}$ on any specified homogeneous element. On the other hand, it is well-known that the classical cyclotomic Hecke algebra of type $A$ is symmetric (\cite{MM}, \cite{BK:HigherSchurWeyl}) and the definition of its symmetrizing form is explicit in that it specifies its value on each monomial basis element. This motivates our second question:

\begin{ques} \label{QB} Can we determine the explicit values of the Shan--Varagnolo--Vasserot symmetrizing form $\Tr^{\text{SVV}}$ on some monomial bases (or at least a set of $K$-linear generators) of the cyclotomic quiver Hecke algebra?
\end{ques}

Note that explicit basis for the center of $\R[\alpha]$ is unknown. Even for the classical cyclotomic Hecke algebra of type $A$, except in the level one case (\cite{GP}) or in the degenerate case (\cite{Brundan:degenCentre}), one does not know any explicit basis for the center as well.

\begin{ques} \label{QC} Can we give an explicit basis for the center of the cyclotomic quiver Hecke algebra?
\end{ques}

The starting points of this paper is to try to answer the above three questions. As a first step toward this goal, we need to consider the case of the cyclotomic quiver Hecke algebra which corresponds to a quiver with a single vertex and no edges. That is, the cyclotomic nilHecke algebra of type $A$. Let us recall its definition.

\begin{dfn} \label{NilHecke} Let $\ell,n\in\N$. The nilHecke algebra $\HH_{n}^{(0)}$ of type $A$ is the unital associative $K$-algebra generated by $\psi_1,\cdots,\psi_{n-1},y_1,\cdots,y_n$
which satisfy the following relations: $$\begin{aligned}
& \psi_r^2=0,\quad \forall\,1\leq r<n,\\
& \psi_r\psi_k=\psi_k\psi_r,\quad\forall\, 1\leq k<r-1<n-1,\\
& \psi_r\psi_{r+1}\psi_r=\psi_{r+1}\psi_r\psi_{r+1},\quad\forall\,1\leq r<n-1,\\
& y_r y_k=y_k y_r,\quad\forall\,1\leq r,k\leq n,\\
& \psi_r y_{r+1}=y_r\psi_r+1,\quad y_{r+1}\psi_r=\psi_r y_r+1,\quad\forall\,1\leq r<n .
\end{aligned}
$$
The cyclotomic nilHecke algebra $\HH_{\ell,n}^{(0)}$ of type $A$ is the quotient of $\HH_{n}^{(0)}$ by the two-sided ideal generated by $y_1^\ell$.
\end{dfn}

The nilHecke algebras $\HH_{n}^{(0)}$ was introduced by Kostant and Kumar \cite{KoK}. It plays an important role in the theory of Schubert calculus, see \cite{Hi}. Mathas has observed that the Specht module over $\HH_{n,n}^{(0)}$ can be realized as the coinvariant algebra with standard bases of Specht modules
being identified with the Schubert polynomials of the coinvariant algebras (\cite[Section 2.5]{Mathas:Singapore}). It is clear that both $\HH_{n}^{(0)}$ and $\HH_{\ell,n}^{(0)}$ are $\Z$-graded $K$-algebras such that each $\psi_r$ is homogeneous with $\deg\psi_r=-2$ and each $y_s$ is homogeneous with $\deg y_s=2$ for all $1\leq r<n, 1\leq s\leq n$.
In \cite[Section 2.5]{Mathas:Singapore}, Mathas has conjectured a monomial basis of the cyclotomic nilHecke algebra $\HH_{n,n}^{(0)}$. In this paper, we shall construct a monomial basis of the cyclotomic nilHecke algebra $\HH_{\ell,n}^{(0)}$ for arbitrary $\ell$ (Theorem \ref{mainthm1}) and in particular verifies Mathas's conjecture. As an application, we shall construct a basis for the center $Z$ of $\HH_{\ell,n}^{(0)}$ (Theorem \ref{mainthm2}). Thus we shall answer Question \ref{QA} and Question \ref{QC} for the cyclotomic nilHecke algebra $\HH_{\ell,n}^{(0)}$. Furthermore, we shall construct a new homogeneous symmetrizing form $\Tr$ (Proposition \ref{keyprop1}) by specifying its values on a homogeneous basis element of $\HH_{\ell,n}^{(0)}$. We prove that this new form $\Tr$ actually coincides with Shan--Varagnolo--Vasserot's symmetrizing form $\Tr^{\text{SVV}}$ on $\HH_{\ell,n}^{(0)}$ introduced in \cite{SVV}. Thus we also answer  Question \ref{QB} for the cyclotomic nilHecke algebra $\HH_{\ell,n}^{(0)}$.

The content of the paper is organised as follows. In Section 2, we shall first review some basic knowledge about the structure and representation of $\HH_{\ell,n}^{(0)}$. Lemma \ref{1relation} provides a useful commutator relation which will be used frequently in later discussion. In Corollary \ref{gradedcharacters} and \ref{gradedCartan} we determine the graded dimensions of the graded simple modules and their graded projective covers as well as the graded decomposition numbers and the graded Cartan numbers. We construct a monomial basis of the cyclotomic nilHecke algebra $\HH_{\ell,n}^{(0)}$ for arbitrary $\ell$ in Theorem \ref{mainthm1}. We also construct a complete set of pairwise orthogonal primitive idempotents in Corollary \ref{matrixUnit1} and Theorem \ref{mainthm0}. In Section 3, we shall first present a basis for the graded basic algebra of $\HH_{\ell,n}^{(0)}$ and show that it is isomorphic to the center $Z$ of $\HH_{\ell,n}^{(0)}$ in Lemma \ref{dimCenter}. Then we shall give a basis for the center in Theorem \ref{mainthm2} which consists of certain symmetric polynomials in $y_1,\cdots,y_n$. We also show in Proposition \ref{matrixIso} that $\HH_{\ell,n}^{(0)}$ is isomorphic to the full matrix algebra over $Z$. In Section 4, we shall first show in Lemma \ref{Zsymmetric} that the center $Z$ is a graded symmetric algebra by specifying an explicit homogeneous symmetrizing form on $Z$. Then we shall introduce two homogeneous symmetrizing forms: one is defined by using its isomorphism with the full matrix algebra over the center $Z$ (Lemma \ref{1stForm}); another is defined by specifying its values on a homogeneous basis element (Definition \ref{DefTr} and Proposition \ref{keyprop1}). We show in Proposition \ref{Comparing2Traces} that these two symmetrizing forms are the same. In Section 5 we show that the form $\Tr$ also coincides with Shan--Varagnolo--Vasserot's symmetrizing form $\Tr^{\text{SVV}}$ (which was introduced in \cite{SVV} for general cyclotomic quiver Hecke algebras).\medskip

After the submission of this paper, Professor Lauda emailed us that he wonders if our results have some connections with his papers \cite{KLMS} and \cite{Lau2}. In \cite{Lau2} he proved that the cyclotomic nilHecke algebra is isomorphic to the matrix ring of size $n!$ over the cohomology of a Grassmannian.
Combining it with  Proposition \ref{matrixIso} in this paper this implies that the center of the cyclotomic nilHecke algebra is isomorphic to that cohomology of a Grassmannian. He also proposed an interesting question of comparing the trace form $\Tr$ in this paper with the natural form on the matrix ring over the cohomology of the Grassmannian which can be defined using integration over the volume form.

\bigskip

\section*{Acknowledgements}

Both author were supported by the National Natural Science Foundation of China (No. 11525102, 11471315). They thank Prof. Lauda and Dr. Kai Zhou for some  helpful discussions and comments.
\bigskip

\bigskip
\section{The structure and representation of $\HH_{\ell,n}^{(0)}$}

Let $\Sym_n$ be the symmetric group of on $\{1,2,\cdots,n\}$ and let $s_i:=(i,i+1)\in\Sym_n$, for $1\leq i<n$. Then $\{s_1,\cdots,s_{n-1}\}$ is the standard set of Coxeter generators for $\Sym_n$. If $w\in\Sym_n$ then the length of $w$ is $$
\ell(w):=\min\{k\in\N|\text{$w=s_{i_1}\cdots s_{i_k}$ for some $1\leq i_1,\cdots,i_k<n$}\}.
$$
If $w=s_{i_1}\cdots s_{i_k}$ with $k=\ell(w)$ then $s_{i_1}\cdots s_{i_k}$ is a reduced expression for $w$. In this case, we define $\psi_w:=\psi_{i_1}\cdots\psi_{i_k}$. The braid relation in Definition \ref{NilHecke} ensures that $\psi_w$ does not depend on the choice of the reduced expression of $w$. Let $w_{0,n}$ be the unique longest element in $\Sym_n$. When $n$ is clear from the context we shall write $w_0$ instead of $w_{0,n}$ for simplicity. Then $w_0=w_0^{-1}$ and $\ell(w_0)=n(n-1)/2$. Let $\ast$ be the unique $K$-algebra anti-automorphism of $\HH_{\ell,n}^{(0)}$ which fixes each of its $\psi$ and $y$ generators.

\begin{lem}\text{(\cite{Man})} \label{nilHeckeCenter} The elements in the set $\{\psi_wy_1^{c_1}\cdots y_n^{c_n}|w\in\Sym_n, c_1,\cdots,c_n\in\N\}$ form a $K$-basis of the nilHecke algebra $\HH_{n}^{(0)}$  and the center of  $\HH_{n}^{(0)}$ is the set of symmetric polynomials in $y_1,\cdots,y_n$.
\end{lem}

Let $\pi: \HH_{n}^{(0)}\twoheadrightarrow\HH_{\ell,n}^{(0)}$ be the canonical surjective homomorphism.

\begin{dfn} An element $z$ in $\HH_{\ell,n}^{(0)}$ is said to be symmetric if $z=\pi(f(y_1,\cdots,y_n))$ for some symmetric polynomial $f(t_1,\cdots,t_n)\in K[t_1,\cdots,t_n]$, where $t_1,\cdots,t_n$ are $n$ indeterminates over $K$.
\end{dfn}

\begin{cor} \label{corcenter1} Any symmetric element in $\HH_{\ell,n}^{(0)}$ lies in the center of $\HH_{\ell,n}^{(0)}$.
\end{cor}

\begin{proof} This follows from \ref{nilHeckeCenter} and the surjective homomorphism $\pi$.
\end{proof}

Let $\Gamma$ be a quiver without loops and $I$ its vertex set. For any $i,j\in I$ let $d_{ij}$ be the number of arrow $i\rightarrow j$ and set $m_{ij}:=d_{ij}+d_{ji}$. This defines a {\it symmetric} generalized Cartan matrix $(a_{ij})_{i,j\in I}$ by putting $a_{ij}:=-m_{ij}$ for $i\neq j$ and
$a_{ii}:=2$ for any $i\in I$. Let $u, v$ be two indeterminates over $\Z$. We define $Q_{ij}:=(-1)^{d_{ij}}(u-v)^{m_{ij}}$ for any $i\neq j\in I$ and $Q_{ii}(u,v):=0$ for any $i\in I$. Let $(\mathfrak{h},\Pi,\Pi^{\vee})$ be a realization of the generalized Cartan matrix $(a_{ij})_{i,j\in I}$. Let $P$ be the associated weight lattice which is a finite rank free abelian group and contains $\Pi=\{\alpha_i|i\in I\}$, let $P^{\vee}$ be the associated co-weight lattice which is a finite rank free abelian group too and contains $\Pi^{\vee}=\{\alpha_i^{\vee}|i\in I\}$. Let $Q^{+}:=\mathbb{N}\Pi\subset P$ be the semigroup generated by $\Pi$ and $P^{+}\subset P$ be the set of integral dominant weights. Let $\Lam\in P^+$ and $\beta\in Q_n^+$. One can associate it with a quiver Hecke algebra $\RR_{\beta}$ as well as its cyclotomic quotient $\R[\beta]$, we refer the readers to \cite{KhovLaud:diagI}, \cite{Rou0} and \cite{SVV} for precise definitions.

Let $\{\Lambda_i|i\in I\}$ be the set of fundamental weights. The nilHecke algebra and its cyclotomic quotient can be regarded as a special quiver Hecke algebra and cyclotomic quiver Hecke algebra. That is, the quiver with single one vertex $\{0\}$ and no edges. More precisely, we have that \begin{equation}\label{Identify2}
\HH_{n}^{(0)}=\RR_{n\alpha_0},\quad \HH_{\ell,n}^{(0)}=\RR^{\ell\Lam_0}_{n\alpha_0} .
\end{equation}

Throughout this paper, unless otherwise stated, we shall work in the category of $\Z$-graded $\HH_{\ell,n}^{(0)}$-modules. Note that $\HH_{\ell,n}^{(0)}$ is a special type $A$ cyclotomic quiver Hecke algebra so that we can apply the theory of graded cellular algebras developed in
\cite{HuMathas:GradedCellular}. We now recall the definition of graded cellular basis in this special situation (i.e., for $\HH_{\ell,n}^{(0)}$).

We use $\emptyset$ to denote the empty partition and $(1)$ to denote the unique partition of $1$. Set $|\emptyset|:=0$, $|(1)|:=1$. We define $$
\P_0:=\bigl\{\blam:=(\lam^{(1)},\cdots,\lam^{(\ell)})\bigm|\sum_{i=1}^{\ell}|\lam^{(i)}|=n, \lam^{(i)}\in\{\emptyset,(1)\},\,\forall\,1\leq i\leq\ell\bigr\} .
$$

\begin{dfn} If $\blam=(\lam^{(1)},\cdots,\lam^{(\ell)})\in\P_0$, then we define $\theta(\blam)$ to be the unique $n$-tuple $(k_1,\cdots,k_n)$ such that
$1\leq k_1<k_2<\cdots<k_n\leq\ell$ and
$$
\lam^{(j)}=\begin{cases} (1), &\text{if $j=k_i$ for some $1\leq i\leq n$,}\\
\emptyset, &\text{otherwise.}
\end{cases}
$$
\end{dfn}
Given any two $n$-tuples $(k_1,\cdots,k_n), (k'_1,\cdots,k'_n)$ of increasing positive integers, we define $$
(k_1,\cdots,k_n)\geq (k'_1,\cdots,k'_n)\Leftrightarrow k_i\geq k'_i,\,\forall\,1\leq i\leq n ,
$$
and $(k_1,\cdots,k_n)>(k'_1,\cdots,k'_n)$ if $(k_1,\cdots,k_n)\geq (k'_1,\cdots,k'_n)$ and $(k_1,\cdots,k_n)\neq (k'_1,\cdots,k'_n)$. For any $\blam,\bmu\in\P_0$, we define $$
\blam>\bmu \Leftrightarrow \theta(\blam)<\theta(\bmu) .
$$
Then ``$>$" is a partial order on $\P_0$. The following definition is a special case of the definition of \cite[Definition 4.15]{HuMathas:GradedCellular}.

\begin{dfn}\text{(\cite[Definition 4.15]{HuMathas:GradedCellular})} Let $\blam\in\P_0$ with $\theta(\blam)=(k_1,\cdots,k_n)$. We define $$
y_\blam:=y_1^{\ell-k_1}\cdots y_n^{\ell-k_n},\quad \deg y_\blam:=2\ell n-2\sum_{i=1}^{n}k_i.
$$
\end{dfn}
By the main results in \cite{HuMathas:GradedCellular}, the elements in the following set \begin{equation}\label{HMcellular}
\bigl\{\psi_{w,u}^{\blam}:=\psi_{w}^\ast y_\blam\psi_u\bigm|\blam\in\P_0, w,u\in\Sym_n\bigr\}
\end{equation}
form a graded cellular $K$-basis of $\HH_{\ell,n}^{(0)}$. Each basis element $\psi_{w,u}^{\blam}$ is homogeneous with degree equal to $$
\deg\psi_{w,u}^{\blam}:=\deg y_\blam-2\ell(w)-2\ell(u)=2\ell n-2\sum_{i=1}^{n}k_i-2\ell(w)-2\ell(u) .
$$

In particular, $\dim_K\HH_{\ell,n}^{(0)}=\ell(\ell-1)\cdots(\ell-n+1) n!$. Note that $\P_0\neq\emptyset$ if and only if $\ell\geq n$. Therefore, $\HH_{\ell,n}^{(0)}=0$ whenever $\ell<n$. Henceforth, we always assume that $\ell\geq n$.

By the general theory of (graded) cellular algebras (\cite{GL}, \cite{HuMathas:GradedCellular}), for each $\blam\in\P_0$, we have a graded Specht module
$S^\blam$, which was equipped with an associative homogeneous bilinear form $\<-,-\>_\blam$. Let $\rad\<-,-\>_\blam$ be the radical of that bilinear form. We define
$D^\blam:=S^\blam/\rad\<-,-\>_\blam$. By \cite[Corollary 5.11]{HuMathas:GradedCellular}, we know that $D^\blam\neq 0$ if and only if $\blam$ is a Kleshchev multipartition with respect to $(p;0,0,\cdots,0)$, where $p=\cha K$.

Let $\blam\in\P_0$ with $\theta(\blam)=(k_1,\cdots,k_n)$. An $\blam$-tableau is a bijection $\t: \{k_1,\cdots,k_n\}\rightarrow\{1,2,\cdots,n\}$. We use $\Std(\blam)$ to denote the set of $\blam$-tableaux. For any $\t\in\Std(\blam)$, we define $$\begin{aligned}
\deg\t&:=\sum_{i=1}^{n}\Bigl(\#\bigl\{k_i<j\leq\ell\bigm|\text{either $j\not\in\{k_1,\cdots,k_n\}$ or $j=k_b$ with $\t(j)>\t(k_i)$}\bigr\}-\\
&\qquad\qquad \#\bigl\{k_i<j\leq\ell\bigm|\text{$j\in\{k_1,\cdots,k_n\}$ and $\t(j)<\t(k_i)$}\bigr\}\Bigr).
\end{aligned}$$
It is clear that in our special case (i.e., for $\P_0$) the above definition of $\deg\t$ coincides with that in \cite{BKW:GradedSpecht} and \cite{HuMathas:GradedCellular}.

\begin{dfn} \label{twoExtremel} We define $$
\blam_{\maxim}:=\bigl(\underbrace{(1),\cdots,(1)}_{\text{$n$ copies}},\underbrace{\emptyset,\cdots,\emptyset}_{\text{$\ell-n$ copies}}\bigr),
\quad\blam_{\minum}:=\bigl(\underbrace{\emptyset,\cdots,\emptyset}_{\text{$\ell-n$ copies}},\underbrace{(1),\cdots,(1)}_{\text{$n$ copies}}\bigr) .
$$
\end{dfn}

It is clear that for any $\bmu\in\P_0\setminus\{\blam_{\maxim},\blam_{\minum}\}$, we have that \begin{equation}\label{maxminum}
\blam_{\minum}<\bmu<\blam_{\maxim},\quad \deg y_{\blam_{\minum}}<\deg y_{\bmu}<\deg y_{\blam_{\maxim}} .
\end{equation}
Using \cite{BK:GradedKL} and the definition of Kleshchev multipartition in \cite{AM}, it is clear that $\blam_{\minum}$
is the unique Kleshchev multipartition in $\P_0$. Therefore, for any $\blam\in\P_0$, $D^\blam\neq 0$ if and only if $\blam=\blam_{\minum}$. Furthermore, $D^{\blam_{\minum}}$ is the unique (self-dual) graded simple module for $\HH_{\ell,n}^{(0)}$. Let $P^{\blam_{\minum}}$ be its graded projective cover.

\begin{dfn} We define $$
D_0:=D^{\blam_{\minum}},\quad P_0:=P^{\blam_{\minum}} .
$$
\end{dfn}

For each $\bmu\in\P_0$, we use $(\HH_{\ell,n}^{(0)})^{>\bmu}$ to denote the $K$-subspace of $\HH_{\ell,n}^{(0)}$ spanned by all the elements of the form
$\psi_w^{\ast}y_\blam\psi_{u}$, where $\blam>\bmu$, $w,u\in\Sym_n$. Then $(\HH_{\ell,n}^{(0)})^{>\bmu}$ is a two-sided ideal of $\HH_{\ell,n}^{(0)}$. By \cite[Corollary 3.11]{HuMathas:GradedInduction}, for any $1\leq r\leq n$, if $\theta(\bmu)=(k_1,\cdots,k_n)$ then \begin{equation}\label{dominanceBigger}
y_\bmu y_r=y_1^{\ell-k_1}\cdots y_n^{\ell-k_n}y_r\in (\HH_{\ell,n}^{(0)})^{>\bmu} .
\end{equation}

\begin{lem} \label{1relation} For any $1\leq i\leq n, 1\leq j<n$, there exists elements $h_{i,j},h'_{i,j}\in\HH_{\ell,n}^{(0)}$ such that \begin{equation}\label{commuRel1}
\psi_{w_0}y_1^{n-1}y_2^{n-2}\cdots y_{n-1}=(-1)^{n(n-1)/2}+\sum_{\substack{1\leq i\leq n\\ 1\leq j<n}}y_ih_{i,j}\psi_{j}.
\end{equation}
Similarly, we have that \begin{equation}\label{commuRel2}
y_1^{n-1}y_2^{n-2}\cdots y_{n-1}\psi_{w_0}=(-1)^{n(n-1)/2}+\sum_{\substack{1\leq i\leq n\\ 1\leq j<n}}\psi_{j}h_{i,j}^{\ast}y_i .
\end{equation}
\end{lem}

\begin{proof} We only prove the first equality as the second one follows from the first one by applying the anti-involution $\ast$. We use induction on $n$. If $n=1$, it is clear that (\ref{commuRel1}) holds. Suppose that the lemma  holds for the nilHecke algebra $\HH_{\ell,n-1}^{(0)}$. We are going to prove (\ref{commuRel1}) for $\HH_{\ell,n}^{(0)}$.

Recall that the unique longest element $w_0:=w_{0,n}$ of $\Sym_n$ has a reduced expression $$
w_0=s_1(s_2s_1)\cdots (s_{n-2}s_{n-3}\cdots s_1)(s_{n-1}s_{n-2}\cdots s_1).
$$
Recall that $w_{0,n-1}$ denotes the unique longest element in $\Sym_{n-1}$ and $w_0=w_{0,n-1}(s_{n-1}s_{n-2}\cdots s_1)$ and $s_1(s_2s_1)\cdots (s_{n-2}s_{n-3}\cdots s_1)$ is a reduce expression for $w_{0,n-1}$.

We define $$
J_n:=\sum_{i=1}^{n}y_i\HH_{\ell,n}^{(0)} .
$$
Then we have that $$\begin{aligned}
&\quad\,\psi_{w_0}y_1^{n-1}y_2^{n-2}\cdots y_{n-1}\\
&=\psi_{w_0}(y_1y_2\cdots y_{n-1})y_1^{n-2}y_2^{n-3}\cdots y_{n-2}\\
&=\psi_{w_{0,n-1}}\bigl(\psi_{n-1}\psi_{n-2}\cdots \psi_1y_1y_2\cdots y_{n-1}\bigr)y_1^{n-2}y_2^{n-3}\cdots y_{n-2}\\
&=\psi_{w_{0,n-1}}\bigl(\psi_{n-1}y_1y_2\cdots y_{n-1}\psi_{n-2}\cdots \psi_1\bigr)y_1^{n-2}y_2^{n-3}\cdots y_{n-2},\qquad\qquad\qquad\quad\text{(By Corollary \ref{corcenter1})}\\
&=\psi_{w_{0,n-1}}\bigl(y_1y_2\cdots y_{n-2}\psi_{n-1}y_{n-1}\psi_{n-2}\cdots \psi_1\bigr)y_1^{n-2}y_2^{n-3}\cdots y_{n-2}\\
&=\psi_{w_{0,n-1}}\Bigl(y_1y_2\cdots y_{n-2}\bigl(y_n\psi_{n-1}-1\bigr)\psi_{n-2}\cdots \psi_1\Bigr)y_1^{n-2}y_2^{n-3}\cdots y_{n-2}\\
&\equiv -\psi_{w_{0,n-1}}\Bigl(y_1y_2\cdots y_{n-2}\psi_{n-2}\cdots \psi_1\Bigr)y_1^{n-2}y_2^{n-3}\cdots y_{n-2}\pmod{J_n}\qquad\qquad\text{(By (\ref{dominanceBigger}))}\\
&\equiv -\psi_{w_{0,n-2}}\Bigl(\psi_{n-2}\psi_{n-3}\cdots\psi_1y_1y_2\cdots y_{n-2}\Bigr)(\psi_{n-2}\cdots \psi_1)y_1^{n-2}y_2^{n-3}\cdots y_{n-2}\pmod{J_n}\\
&\equiv -\psi_{w_{0,n-2}}\Bigl(\psi_{n-2}y_1y_2\cdots y_{n-2}\psi_{n-3}\cdots\psi_1\Bigr)(\psi_{n-2}\cdots \psi_1)y_1^{n-2}y_2^{n-3}\cdots y_{n-2}\pmod{J_n}\\
&\equiv -\psi_{w_{0,n-2}}\Bigl(y_1y_2\cdots y_{n-3}(\psi_{n-2}y_{n-2})\psi_{n-3}\cdots\psi_1\Bigr)(\psi_{n-2}\cdots \psi_1)y_1^{n-2}y_2^{n-3}\cdots y_{n-2}\pmod{J_n}\\
&\equiv -\psi_{w_{0,n-2}}\Bigl(y_1y_2\cdots y_{n-3}(y_{n-1}\psi_{n-2}-1)\psi_{n-3}\cdots\psi_1\Bigr)(\psi_{n-2}\cdots \psi_1)y_1^{n-2}y_2^{n-3}\cdots y_{n-2}\pmod{J_n}\\
&\equiv (-1)^2\psi_{w_{0,n-2}}\Bigl(y_1y_2\cdots y_{n-3}\psi_{n-3}\cdots\psi_1\Bigr)(\psi_{n-2}\cdots \psi_1)y_1^{n-2}y_2^{n-3}\cdots y_{n-2}\pmod{J_n}\\
&\equiv (-1)^2\psi_{w_{0,n-2}}\bigl(y_1y_2\cdots y_{n-3}\bigr)\bigl((\psi_{n-3}\cdots\psi_1)(\psi_{n-2}\cdots \psi_1)\bigr)y_1^{n-2}y_2^{n-3}\cdots y_{n-2}\pmod{J_n}\\
&\equiv (-1)^2\psi_{w_{0,n-3}}\bigl(\psi_{n-3}\psi_{n-4}\cdots\psi_1y_1y_2\cdots y_{n-3}\bigr)\Bigl((\psi_{n-3}\cdots\psi_1)(\psi_{n-2}\cdots \psi_1)(y_1^{n-2}y_2^{n-3}\cdots y_{n-2})\Bigr)\pmod{J_n}\\
&\quad\,\vdots\\
&\equiv (-1)^{n-1}\Bigl(\psi_1(\psi_2\psi_1)\cdots(\psi_{n-3}\cdots\psi_1)(\psi_{n-2}\cdots \psi_1)(y_1^{n-2}y_2^{n-3}\cdots y_{n-2})\Bigr)\pmod{J_n}\\
&\equiv (-1)^{n-1}\psi_{w_{0,n-1}}(y_1^{n-2}y_2^{n-3}\cdots y_{n-2})\pmod{J_n}\\
&\equiv (-1)^{n-1}(-1)^{(n-1)(n-2)/2}\equiv (-1)^{n(n-1)/2}\pmod{J_n},
\end{aligned}
$$
as required, where we use induction in the second last equality.

Therefore, we have proved that $$
\psi_{w_0}y_1^{n-1}y_2^{n-2}\cdots y_{n-1}=(-1)^{n(n-1)/2}+\sum_{\substack{1\leq i\leq n\\ 1\leq j<n}}y_ih_i,
$$
where $h_i\in\HH_{\ell,n}^{(0)}$. Comparing the degree on both sides, we can assume that each $h_i$ is homogeneous with $h_i\neq 0$ only if $\deg(h_i)=-2<0$.
On the other hand, we can express each nonzero $h_i$ as a $K$-linear combination of some monomials of the form $y_1^{c_1}\cdots y_n^{c_n}\psi_w$, where
$c_1,\cdots,c_n\in\N, w\in\Sym_n$. Since each $y_j$ has degree $2$, we can thus deduce that each nonzero $h_i$ must be equal to a $K$-linear combination of some monomials of the form $y_1^{c_1}\cdots y_n^{c_n}\psi_w$ with $c_1,\cdots,c_n\in\N$ and $1\neq w\in\Sym_n$. This completes the proof of the lemma.
\end{proof}

\begin{lem}\label{2usefulProperties} 1) For any $u,w\in\Sym_n$, if $\ell(u)+\ell(w)>\ell(uw)$, then $\psi_u\psi_w=0$;

2) For any $1\leq r<n$, $\psi_r\psi_{w_0}=0=\psi_{w_0}\psi_r$.
\end{lem}

\begin{proof} 1) follows from the defining relations for  $\HH_{\ell,n}^{(0)}$, while 2) follows from the defining relations for  $\HH_{\ell,n}^{(0)}$ and the fact that $w_0$ has both a reduced expression which starts with $s_r$ as well as a reduced expression which ends with $s_r$ for any $1\leq r<n$.
\end{proof}

Let $s\in\Z$. For any $\Z$-graded $\HH_{\ell,n}^{(0)}$-module $M$, we define $M\<s\>$ to be a new $\Z$-graded $\HH_{\ell,n}^{(0)}$-module as follows: 1) $M\<s\>=M$ as ungraded $\HH_{\ell,n}^{(0)}$-module; 2) As a $\Z$-graded module, $M\<s\>$ is obtained by shifting the grading on~$M$ up by~$s$. That is,
$M\<s\>_d=M_{d-s}$, for $d\in\Z$.

\begin{lem} \label{bili} Let $\bmu\in\P_0$ with $\theta(\bmu)=(k_1,\cdots,k_n)$. Then $$\dim D_0=n!,\quad \dim P_0=\bigl(^\ell_n\!\bigr)n!,\quad
S^\bmu\cong D_0\<n\ell-\frac{n(n-1)}{2}-\sum_{i=1}^{n}k_i\> .
$$
\end{lem}

\begin{proof} By the definitions of $\P_0$ and Specht modules over $\HH_{\ell,n}^{(0)}$, it is clear that $S^\bmu\cong S^{\blam_{\minum}}\<n\ell-\frac{n(n-1)}{2}-\sum_{i=1}^{n}k_i\>$. Thus it suffices to show that $S^{\blam_{\minum}}=D^{\blam_{\minum}}$. To this end, we need to compute the bilinear form between standard bases of the Specht module $S^{\blam_{\minum}}$.

By definition,  $S^{\blam_{\minum}}$ has a standard basis $$
\{y_1^{n-1}y_2^{n-2}\cdots y_{n-1}\psi_w+(\HH_{\ell,n}^{(0)})^{>\blam_{\minum}}|w\in\Sym_n\} .
$$

For any $w,u\in\Sym_n$, by Lemma \ref{2usefulProperties}, we see that $$
y_1^{n-1}y_2^{n-2}\cdots y_{n-1}\psi_w\psi_u^\ast y_1^{n-1}y_2^{n-2}\cdots y_{n-1}=y_1^{n-1}y_2^{n-2}\cdots y_{n-1}(\psi_w\psi_{u^{-1}})y_1^{n-1}y_2^{n-2}\cdots y_{n-1}=0
$$
unless $\ell(wu^{-1})=\ell(w)+\ell(u^{-1})$.

Now we assume that $\ell(wu^{-1})=\ell(w)+\ell(u^{-1})$. By the commutator relations between $y$ and $\psi$ generators, (\ref{dominanceBigger}) and the fact that $\ell(w_0)=n(n-1)2$, we can deduce that $$
y_1^{n-1}y_2^{n-2}\cdots y_{n-1}(\psi_w\psi_{u^{-1}})y_1^{n-1}y_2^{n-2}\cdots y_{n-1}=y_1^{n-1}y_2^{n-2}\cdots y_{n-1}\psi_{wu^{-1}}y_1^{n-1}y_2^{n-2}\cdots y_{n-1}\in(\HH_{\ell,n}^{(0)})^{>\blam_{\minum}}
$$
unless $wu^{-1}=w_0$. In that case, by Lemma \ref{1relation}, we have that $$\begin{aligned}
y_1^{n-1}y_2^{n-2}\cdots y_{n-1}\psi_w\psi_u^\ast y_1^{n-1}y_2^{n-2}\cdots y_{n-1}
&=y_1^{n-1}y_2^{n-2}\cdots y_{n-1}\psi_{w_0}y_1^{n-1}y_2^{n-2}\cdots y_{n-1}\\
&=(-1)^{n(n-1)/2}y_1^{n-1}y_2^{n-2}\cdots y_{n-1}\pmod{(\HH_{\ell,n}^{(0)})^{>\blam_{\minum}}}.
\end{aligned}
$$
Thus we have proved that if $\ell(wu^{-1})=\ell(w)+\ell(u^{-1})$ and $wu^{-1}=w_0$, then
$$
\<y_1^{n-1}y_2^{n-2}\cdots y_{n-1}\psi_w+(\HH_{\ell,n}^{(0)})^{>\blam_{\minum}},y_1^{n-1}y_2^{n-2}\cdots y_{n-1}\psi_u+(\HH_{\ell,n}^{(0)})^{>\blam_{\minum}}\>_{\blam_{\minum}}=(-1)^{n(n-1)/2} ;
$$
otherwise it is equal to $0$. This means the Gram matrix of $S^{\blam_{\minum}}$ is invertible and hence the bilinear form $\<-,-\>_{\blam_{\minum}}$ on $S^{\blam_{\minum}}$ is non-degenerate. It follows that $S^{\blam_{\minum}}=D^{\blam_{\minum}}=D_0$ as required.
Therefore, $\dim D_0=\dim S^{\blam_{\minum}}=n!$. Finally, since $\HH_{\ell,n}^{(0)}\cong P_0^{\oplus\dim D_0}$, we can deduce that $\dim P_0=\dim\HH_{\ell,n}^{(0)}/\dim D_0=\bigl(^\ell_n\!\bigr)(n!)^2/n!=\bigl(^\ell_n\!\bigr)n!$.
\end{proof}

Let $q$ be an indeterminate. The \textbf{graded dimension} of $M$ is the
Laurent polynomial
\begin{equation}\label{E:GrDim}
\dim_q M=\sum_{d\in\Z}(\dim_K M_d)\,q^d\in\N[q,q^{-1}],
\end{equation}
where $M_d$ is the homogeneous component of $M$ which has degree $d$. In particular, $\dim_K M=(\dim_q M)\!\bigm|_{q=1}$.
As a consequence, we can determine the graded dimension of the unique self-dual graded simple module $D_0$ and its projective cover $P_0$, and compute the graded decomposition number $d_{\bmu,\blam_{\minum}}(q):=[S^\bmu:D^{\blam_{\minum}}]_q$ and graded Cartan number $c_{\blam_{\minum},\blam_{\minum}}(q):=[P^{\blam_{\minum}}:D^{\blam_{\minum}}]_q$.

\begin{cor} \label{gradedcharacters} We have that $$
\dim_q D_0=\sum_{\t\in\Std(\blam_{\minum})}q^{\deg\t},\quad \dim_q P_0=\sum_{\substack{k_1,\cdots,k_n\in\N\\ 1\leq k_1<k_2<\cdots<k_n\leq\ell
}}\sum_{\t\in\Std(\blam_{\minum})}q^{\deg\t+2n\ell-n(n-1)-\sum_{i=1}^{n}2k_i} .
$$
\end{cor}

\begin{cor} \label{gradedCartan} Let $\bmu\in\P_0$ with $\theta(\bmu)=(k_1,\cdots,k_n)$. We have that $$\begin{aligned}
d_{\bmu,\blam_{\minum}}(q)&=q^{n\ell-\frac{n(n-1)}{2}-\sum_{i=1}^{n}k_i}\in\delta_{\bmu,\blam_{\minum}}+q\N[q],\\
c_{\blam_{\minum},\blam_{\minum}}(q)&=\sum_{\substack{l_1,\cdots,l_n\in\N\\ 1\leq l_1<l_2<\cdots<l_n\leq\ell
}}q^{2n\ell-n(n-1)-\sum_{i=1}^{n}2l_i}\in 1+q\N[q] .
\end{aligned}
$$
\end{cor}

\begin{lem}\text{(\cite[Proposition 7]{HoffnungLauda:KLRnilpotency})} \label{HLlem0} For any $1\leq s\leq n$, we have that $$
\sum_{\substack{l_1,\cdots,l_s\in\N\\ l_1+\cdots+l_s=\ell-s+1}}y_1^{l_1}y_2^{l_2}\cdots y_s^{l_s}=0 .
$$
\end{lem}

\begin{rem} Note that one should identify our generator $y_r$ with the generator $-x_{r,\bi}$ in \cite{HoffnungLauda:KLRnilpotency} so that the relation $\psi_r y_{r+1}=y_r\psi_r+1$ in Definition \ref{NilHecke} matches up with the relation $x_{r,\bi}\delta_{r,i}-\delta_{r,\bi}x_{r+1,\bi}=e(\bi)$ when $i_r=i_{r+1}$.
\end{rem}

\begin{lem}\text{(\cite[Proposition 8]{HoffnungLauda:KLRnilpotency})} \label{HLlem} Let $1\leq m<n$ and $b\in\N$. If $y^{b}_{m-1}=0$ then $y^{b}_{m}=0$.
\end{lem}

\begin{lem} \label{vanish1} For any $2\leq m\leq n$ and $\omega_m>\ell-m$, we have that
\begin{equation}\label{step1}
y^{\ell-1}_{1}y^{\ell-2}_{2} \cdots y^{\ell-m+1}_{m-1}y^{\omega_m}_{m}=0,
\end{equation}
\end{lem}

\begin{proof} We use induction on $m$. If $m=1$, then (\ref{step1}) reduces to $y^{\omega_1}_1=0$ for $\omega_1>\ell-1$, which certainly holds by the fact that $y_1^\ell=0$.

If $m=2$, then we need to show that $y^{\ell-1}_{1}y^{\omega_2}_{2}=0$ whenever $\omega_2>\ell-2$. By Lemma \ref{HLlem}, we can deduce that $y^{\ell}_{2}=0$ from the equality $y_1^\ell=0$. Therefore, it remains to show that $y^{\ell-1}_{1}y^{\ell-1}_{2}=0$. In this case, applying Lemma \ref{HLlem0}, we get that
$$
y^{\ell-1}_{2}=\sum_{\substack{l_1,l_2\in\N, l_1\neq 0\\ l_1+l_2=\ell-1}}y^{l_1}_{1}y^{l_2}_{2}.
$$
It follows that $$
y^{\ell-1}_{1}y^{\ell-1}_{2}=-\sum_{\substack{l_1,l_2\in\N, l_1\neq 0\\ l_1+l_2=\ell-1}}y^{\ell-1+l_1}_{1}y^{l_2}_{2}=0,
$$
as required.

Now we assume that (\ref{step1}) holds for any $2\leq k \leq m$. That says, $y^{\ell-1}_{1}y^{\ell-2}_{2} \cdots y^{\ell-k+1}_{k-1}y^{\omega_k}_{k}=0$
whenever $\omega_k>\ell-k$.

Applying Lemma \ref{HLlem0} for $s=m+1$, we get that
$$
y^{\ell-m}_{m+1}=\sum_{\substack{l_1,\cdots,l_{m+1}\in\N\\ l_{m+1}\neq \ell-m, l_1+\cdots +l_{m+1}=\ell-m}}
y^{l_1}_{1}y^{l_2}_{2}\cdots y^{l_{m+1}}_{m+1}.
$$
It follows that for any $\omega_{m+1}> \ell-(m+1)$,
$$
\begin{aligned}
&\quad\,y^{\ell-1}_{1}y^{\ell-2}_{2} \cdots y^{\ell-m+1}_{m-1}y^{\omega_{m+1}}_{m+1}  \\
&=y^{\ell-1}_{1}y^{\ell-2}_{2} \cdots y^{\ell-m+1}_{m-1}y^{\omega_{m+1}-(\ell-m)}_{m+1}y^{\ell-m}_{m+1}  \\
&=-\sum_{\substack{l_2,\cdots,l_{m+1}\in\N\\ l_{m+1}\neq \ell-m, l_2+\cdots +l_{m+1}=\ell-m}}y^{\ell-1}_{1}y^{\ell-2+l_2}_{2}\cdots y^{\omega_{m+1}-(\ell-m)+l_{m+1}}_{m+1} \\
&=-\sum_{\substack{l_m,l_{m+1}\in\N\\ l_{m+1}\neq \ell-m, l_m+l_{m+1}=\ell-m}}y^{\ell-1}_{1}y^{\ell-2}_{2}\cdots y^{\ell-m+1}_{m-1} y^{\ell-m+l_{m}}_{m} y^{\omega_{m+1}-(\ell-m)+l_{m+1}}_{m+1}  \\
&=0.
\end{aligned}
$$
where we have used induction hypothesis in the third and fourth equalities. This completes the proof of the lemma.\end{proof}

\begin{cor} \label{matrixUnit1} For any $z_1,z_2\in\Sym_n$, we define $F_{z_1,z_2}:=(-1)^{n(n-1)/2}\psi_{w_0z_1,z_2}^{\blam_{\minum}}$. Then $F_{z_1,z_2}\neq 0$ is a homogeneous element of degree $2\ell(z_1)-2\ell(z_2)$. Suppose that $\ell=n$. Then $\sum_{w\in\Sym_n}F_{w,w}=1$ and 	
	$$
	F_{z_1,z_2}F_{u_1,u_2}=\delta_{z_2,u_1}F_{z_1,u_2},\,\,\forall\,u_1,u_2\in\Sym_n .
	$$
	In particular, $\HH_{n,n}^{(0)}$ is isomorphic to the full matrix algebra $M_{n!\times n!}(K)$ over $K$ with $\{F_{u,w}\}_{u,w\in\Sym_n}$ being  a complete set of matrix units.	
\end{cor}

\begin{proof} As a cellular basis element, we know that $\psi_{w_0z_1,z_2}^{\blam_{\minum}}\neq 0$ and hence $F_{z_1,z_2}\neq 0$. By definition, it is clear that $F_{z_1,z_2}$ is a homogeneous element of degree $2\ell(z_1)-2\ell(z_2)$.

Suppose that $\ell=n$. By Lemma \ref{vanish1}, for any $1\leq r\leq n$, we have that \begin{equation}\label{rightVanish}
	y_1^{n-1}y_2^{n-2}\cdots y_{n-1}y_r=(y_1^{n-1}y_2^{n-2}\cdots y_{r+1}^{n-r-1}y_r^{n-r+1})y_{r-1}^{n-r+1}y_{r-2}^{n-r+2}\cdots y_{n-1}=0 .
	\end{equation}
For any $u_1,u_2\in\Sym_n$, $$\begin{aligned}
F_{z_1,z_2}F_{u_1,u_2}&=\psi_{w_0z_1}^{\ast}y_1^{n-1}y_2^{n-2}\cdots y_{n-1}\psi_{z_2}\psi_{w_0u_1}^\ast y_1^{n-1}y_2^{n-2}\cdots y_{n-1}\psi_{u_2} .
\end{aligned}
$$
By Lemma \ref{2usefulProperties}, the above equality is zero unless $\ell(z_2(w_0u_1)^{-1})=\ell(z_2)+\ell((w_0u_1)^{-1})$. So we can assume that
$\ell(z_2(w_0u_1)^{-1})=\ell(z_2)+\ell((w_0u_1)^{-1})$. Then we get that $$\begin{aligned}
F_{z_1,z_2}F_{u_1,u_2}&=\psi_{w_0z_1}^{\ast}y_1^{n-1}y_2^{n-2}\cdots y_{n-1}\psi_{z_2u_1^{-1}w_0^{-1}}y_1^{n-1}y_2^{n-2}\cdots y_{n-1}\psi_{u_2} .
\end{aligned}
$$
Note that $w_0$ is the unique longest element in $\Sym_n$ with length $(n-1)n/2$. If $z_2u_1^{-1}w_0^{-1}\neq w_0$ then we must have that $$
\psi_{z_2u_1^{-1}w_0^{-1}}y_1^{n-1}y_2^{n-2}\cdots y_{n-1}\in\sum_{j=1}^{n}y_j\HH_{n,n}^{(0)} .
$$
In that case, $F_{z_1,z_2}F_{u_1,u_2}=0$ by (\ref{rightVanish}). Therefore, we can further assume that $z_2u_1^{-1}w_0^{-1}=w_0$ and hence $z_2=u_1$. In the latter case, $F_{z_1,z_2}F_{u_1,u_2}=F_{z_1,u_2}$ by Lemma \ref{1relation} and (\ref{rightVanish}). This proves the first part of the corollary.

The second part of the corollary follows from Corollary \ref{matrixUnit1} and the fact that $\dim\HH_{n,n}^{(0)}=(n!)^2$ and $\{F_{z_1,z_2}|z_1,z_2\in\Sym_n\}$ is a basis of $\HH_{n,n}^{(0)}$.
\end{proof}

Recall that the weak Bruhat order ``$\succeq$" on $\Sym_n$ is defined as follows (cf. \cite{DJ1}): For $u,w\in\Sym_n$, let $u\succeq w$ if there is a reduced expression
$w=s_{j_1}\cdots s_{j_k}$ for $w$ and $u=s_{j_1}\cdots s_{j_l}$ for some $l\leq k$. We write $u\succ w$ if $u\succeq w$ and $u\neq w$.

\begin{cor} \label{fprimecor} Let $\ell,n\in\N$. For any $z_1,z_2\in\Sym_n$, we define $$
	F'_{z_1,z_2}:=\psi_{w_0z_1}^{\ast}y_1^{n-1}y_2^{n-2}\cdots y_{n-1}\psi_{w_0}y_1^{n-1}y_2^{n-2}\cdots y_{n-1}\psi_{z_2}.
$$
Then $F'_{z_1,z_2}\neq 0$ is a homogeneous element of degree $2\ell(z_1)-2\ell(z_2)$, and $$\begin{aligned}
&(F'_{z_1,z_1})^2=F'_{z_1,z_1},\quad F'_{z_1,z_2}=F'_{z_1,z_1}F'_{z_1,z_2}=F'_{z_1,z_2}F'_{z_2,z_2},\\
& F'_{z_1,z_2}F'_{z_2,u_2}=F'_{z_1,u_2},\,\, F'_{z_1,z_2}F'_{u_1,u_2}=0,\,\,\,\text{$\forall\,u_1,u_2\in\Sym_n$ with $z_2^{-1}\nsucceq u_1^{-1}$} .\\
\end{aligned}$$
\end{cor}

\begin{proof} By Lemma \ref{1relation} and (\ref{dominanceBigger}), we have that \begin{equation}\label{fprime}
F'_{z_1,z_2}\equiv (-1)^{(n-1)n/2}\psi_{w_0z_1,z_2}^{\blam_{\minum}}\pmod{(\HH_{\ell,n}^{(0)})^{>\blam_{\minum}}}.
\end{equation}
In particular, this implies that $F'_{z_1,z_2}\neq 0$ by the cellular structure of $\HH_{\ell,n}^{(0)}$. By definition, it is clear that  $F'_{z_1,z_2}$ is a homogeneous element of degree $2\ell(z_1)-2\ell(z_2)$.

Again by Lemma \ref{1relation} and Lemma \ref{2usefulProperties}, we have that $$\begin{aligned}
(F'_{z_1,z_1})^2&=\psi_{w_0z_1}^{\ast}y_1^{n-1}y_2^{n-2}\cdots y_{n-1}\psi_{w_0}y_1^{n-1}y_2^{n-2}\cdots y_{n-1}(\psi_{z_1}\psi_{w_0z_1}^{\ast})y_1^{n-1}y_2^{n-2}\cdots y_{n-1}\psi_{w_0}y_1^{n-1}y_2^{n-2}\cdots y_{n-1}\psi_{z_1}\\
&=\psi_{w_0z_1}^{\ast}y_1^{n-1}y_2^{n-2}\cdots y_{n-1}\bigl(\psi_{w_0}y_1^{n-1}y_2^{n-2}\cdots y_{n-1}\psi_{w_0}\bigr)y_1^{n-1}y_2^{n-2}\cdots y_{n-1}\psi_{w_0}y_1^{n-1}y_2^{n-2}\cdots y_{n-1}\psi_{z_1}\\
&=(-1)^{(n-1)n/2}\psi_{w_0z_1}^{\ast}y_1^{n-1}y_2^{n-2}\cdots y_{n-1}\bigl(\psi_{w_0}y_1^{n-1}y_2^{n-2}\cdots y_{n-1}\psi_{w_0}\bigr)y_1^{n-1}y_2^{n-2}\cdots y_{n-1}\psi_{z_1}\\
&=F'_{z_1,z_1} .
\end{aligned}
$$
A similar argument also shows that $F'_{z_1,z_2}=F'_{z_1,z_1}F'_{z_1,z_2}=F'_{z_1,z_2}F'_{z_2,z_2}$ and $F'_{z_1,z_2}F'_{z_2,u_2}=F'_{z_1,u_2}$.

Finally, let $u_1,u_2\in\Sym_n$ such that $z_2^{-1}\nsucceq u_1^{-1}$. We have that $$\begin{aligned}
F'_{z_1,z_2}F'_{u_1,u_2}&=\psi_{w_0z_1}^{\ast}y_1^{n-1}y_2^{n-2}\cdots y_{n-1}\psi_{w_0}y_1^{n-1}y_2^{n-2}\cdots y_{n-1}(\psi_{z_2}\psi_{w_0u_1}^{\ast})y_1^{n-1}y_2^{n-2}\cdots y_{n-1}\psi_{w_0}\\
&\qquad y_1^{n-1}y_2^{n-2}\cdots y_{n-1}\psi_{u_2}.
\end{aligned}
$$
Note that the assumption $z_2^{-1}\nsucceq u_1^{-1}$ implies that $\ell(z_2u_1^{-1}w_0^{-1})\neq \ell(z_2)+\ell(u_1^{-1}w_0^{-1})$ because otherwise we would have some $x\in\Sym_n$ such that $xz_2=u_1$ and $$
\ell(x)=\ell(w_0)-\ell(z_2u_1^{-1}w_0^{-1})=\ell(w_0)-(\ell(z_2)+\ell(u_1^{-1}w_0^{-1}))=\ell(w_0)-\ell(z_2)-(\ell(w_0)-\ell(u_1^{-1}))=\ell(u_1)-\ell(z_2).
$$

By Lemma \ref{2usefulProperties}, $\ell(z_2u_1^{-1}w_0^{-1})\neq \ell(z_2)+\ell(u^{-1}w_0^{-1})$ implies that $\psi_{z_2}\psi_{w_0u_1}^{\ast}=0$. We thus proved that
$F'_{z_1,z_2}F'_{u_1,u_2}=0$ as required.
\end{proof}

\begin{dfn} \label{wf} We fix a total order on $\Sym_n$ and list the elements in $\Sym_n$ as $1=w_1,w_2,\cdots,w_{n!}$ such that $$
\text{$w_i^{-1}\succ w_j^{-1}$ only if $i<j$.}
$$
We define a set of elements $\{\wf_{w_i,w_j}|1\leq i,j\leq n!\}$ in $\HH_{\ell,n}^{(0)}$ inductively as follows: $$
\wf_{w_1,w_j}=\wf_{1,w_j}:=F'_{1,w_j},\quad\forall\,1\leq j\leq n! .
$$
Suppose that $\wf_{w_k,w_j}$ was already defined for any $1\leq k<i$ and $1\leq j\leq n!$. Then we define $$
\wf_{w_i,w_j}:=F'_{w_i,w_j}-\sum_{1\leq k<i}\wf_{w_k,w_k}F'_{w_i,w_j},\quad\forall\,1\leq j\leq n! .
$$
\end{dfn}
By construction and Corollary \ref{fprimecor}, we see that \begin{equation}\label{Aproperty}
\wf_{w_i,w_j}F'_{w_j,w_a}=\wf_{w_i,w_a},\,\,\forall\,1\leq a\leq n! .
\end{equation}

\begin{thm} \label{mainthm0} For any $1\leq i,j\leq n!$, we have that $\wf_{w_i,w_j}\neq 0$ is a homogeneous element of degree $2\ell(w_i)-2\ell(w_j)$ and \begin{equation}\label{matrixunit2}
\wf_{w_i,w_j}\wf_{w_k,w_l}=\delta_{j,k}\wf_{w_i,w_l} ,\quad\forall\, 1\leq k,l\leq n! .
\end{equation}
Moreover, for each $1\leq i\leq n!$, $\wf_{w_i,w_i}\HH_{\ell,n}^{(0)}\cong P_0$ as an ungraded right $\HH_{\ell,n}^{(0)}$-module, $1=\sum_{i=1}^{n!}\wf_{w_i,w_i}$, and
$\{\wf_{w_i,w_i}|1\leq i\leq n!\}$ is a complete set of pairwise orthogonal primitive idempotents of $\HH_{\ell,n}^{(0)}$.	
\end{thm}

\begin{proof} By (\ref{fprime}), for any $u\in\Sym_n$ with $u^{-1}\succ w_1^{-1}$, $$\begin{aligned}
F'_{u,u}F'_{w_1,w_2}&\equiv \psi_{w_0u,u}^{\blam_{\minum}}\psi_{w_0w_1,w_2}^{\blam_{\minum}}\pmod{(\HH_{\ell,n}^{(0)})^{>\blam_{\minum}}}	\\
&\equiv \psi_{w_0u}^{\ast}y_1^{n-1}y_2^{n-2}\cdots y_{n-1}\bigl(\psi_{u}(\psi_{w_0w_1})^{\ast}\bigr)y_1^{n-1}y_2^{n-2}\cdots y_{n-1}\psi_{w_2}\pmod{(\HH_{\ell,n}^{(0)})^{>\blam_{\minum}}}\\
&\equiv \psi_{w_0u}^{\ast}y_1^{n-1}y_2^{n-2}\cdots y_{n-1}\psi_{uw_1^{-1}w_0} y_1^{n-1}y_2^{n-2}\cdots y_{n-1}\psi_{w_2}\pmod{(\HH_{\ell,n}^{(0)})^{>\blam_{\minum}}}\\
&\equiv \sum_{j=1}^{n}r_j \psi_{w_0u}^{\ast}y_1^{n-1}y_2^{n-2}\cdots y_{n-1}y_jh_j\psi_{w_2}\pmod{(\HH_{\ell,n}^{(0)})^{>\blam_{\minum}}}\\
&\equiv 0\pmod{(\HH_{\ell,n}^{(0)})^{>\blam_{\minum}}},	\end{aligned} $$
where $r_j\in K$, $h_j\in\HH_{\ell,n}^{(0)}$ for any $z,j$. Combining this with Corollary \ref{fprimecor} and (\ref{fprime}) we can deduce that \begin{equation}\label{leadingterm2}
\wf_{w_i,w_j}\equiv(-1)^{(n-1)n/2}\psi_{w_0w_i,w_j}^{\blam_{\minum}}\pmod{(\HH_{\ell,n}^{(0)})^{>\blam_{\minum}}}.
\end{equation}
In particular,  $\wf_{w_i,w_j}\neq 0$. By definition, Corollary \ref{fprimecor} and an easy induction, it is easy to see that $\wf_{w_i,w_j}$ is a homogeneous element of degree $2\ell(w_i)-2\ell(w_j)$.

We are going to prove (\ref{matrixunit2}). We use induction on $k$. Suppose that $k=1$. If $j\neq 1$, then $j>1$. By construction, $$
\wf_{w_i,w_j}\in\sum_{w\in\Sym_n}\HH_{\ell,n}^{(0)}F'_{w,w_j},\,\,\,\wf_{1,w_l}=F'_{1,w_l} .
$$
By Corollary \ref{fprimecor}, we have that $F'_{w,w_j}F'_{1,u}=0$. It follows that $\wf_{w_i,w_j}\wf_{w_1,w_l}=0$. If $j=1$, then by (\ref{Aproperty}) we have that $$
\wf_{w_i,w_1}\wf_{w_1,w_l}=\wf_{w_i,1}F'_{1,w_l}=\wf_{w_i,w_l} .
$$
as required.

In general, suppose that (\ref{matrixunit2}) holds for any $k<m$. Let us consider the case when $k=m$. By construction, we have that
 $$
 \wf_{w_i,w_j}\in\sum_{w\in\Sym_n}\HH_{\ell,n}^{(0)}F'_{w,w_j},\,\,\,\wf_{w_m,w_l}\in\sum_{\substack{u\in\Sym_n\\ 1\leq a\leq m}}F'_{w_a,u}\HH_{\ell,n}^{(0)} .
 $$
Therefore, if $j>m$ then $\wf_{w_i,w_j}\wf_{w_m,w_l}=0$ by Corollary \ref{fprimecor}.

Suppose that $j<m$. Then $$\begin{aligned}
\wf_{w_i,w_j}\wf_{w_m,w_l}&=\wf_{w_i,w_j}\bigl(F'_{w_m,w_l}-\sum_{1\leq k<m}\wf_{w_k,w_k}F'_{w_m,w_l}\bigr)\\
&=\wf_{w_i,w_j}\bigl(F'_{w_m,w_l}-\sum_{1\leq k<m}\delta_{k,j}\wf_{w_k,w_k}F'_{w_m,w_l}\bigr)\\
&=\wf_{w_i,w_j}F'_{w_m,w_l}-\wf_{w_i,w_j}F'_{w_m,w_l}\\
&=0,\end{aligned}
$$
as required, where we have used induction hypothesis in the second and the third equalities.

Suppose that $j=m$. Then $$\begin{aligned}
\wf_{w_i,w_m}\wf_{w_m,w_l}&=\wf_{w_i,w_m}\bigl(F'_{w_m,w_l}-\sum_{1\leq k<m}\wf_{w_k,w_k}F'_{w_m,w_l}\bigr)\\
&=\wf_{w_i,w_m}F'_{w_m,w_l}-\sum_{1\leq k<m}\wf_{w_i,w_m}\wf_{w_k,w_k}F'_{w_m,w_l}\\
&=\wf_{w_i,w_m}F'_{w_m,w_l}-0=\wf_{w_i,w_l},\end{aligned}
$$
as required, where we used (\ref{Aproperty}) in the last equality, and used the induction hypothesis in the second last equality.

Since $$P_0^{\oplus\dim D_0}=P_0^{\oplus n!}\cong\HH_{\ell,n}^{(0)}\cong\bigl(1-\sum_{w\in\Sym_n}\wf_{w,w}\bigr)\HH_{\ell,n}^{(0)}\bigoplus\Bigl(\oplus_{w\in\Sym_n}\wf_{w,w}\HH_{\ell,n}^{(0)}\Bigr),
$$
and $\wf_{w,w}\HH_{\ell,n}^{(0)}\neq 0$ for each $w\in\Sym_n$. By Krull-Schmidt theorem we can deduce that for each $w\in\Sym_n$, $F_{w,w}\HH_{\ell,n}^{(0)}\cong P_0$ as ungraded right $\HH_{\ell,n}^{(0)}$-module and $1=\sum_{w\in\Sym_n}\wf_{w,w}$.	In other words,
$\{\wf_{w_i,w_i}|1\leq i\leq n!\}$ is a complete set of pairwise orthogonal primitive idempotents of $\HH_{\ell,n}^{(0)}$.	This completes the proof of the theorem.
\end{proof}

The following result was first conjectured by A. Mathas \cite[\S2.5, before Corollary 2.5.2]{Mathas:Singapore} in the special case when $\ell=n$.

\begin{thm} \label{mainthm1} The elements in the following set \begin{equation}\label{AKbasis}
\bigl\{\psi_wy_1^{a_1}\cdots y_n^{a_n}\bigm|0\leq a_i\leq \ell-i,\,\,\forall\,1\leq i\leq n, w\in\Sym_n\bigr\}
\end{equation}
form a $K$-basis of $\HH_{\ell,n}^{(0)}$.
\end{thm}

\begin{proof} We first claim that for any $b_1,\cdots,b_{m-1},\omega_m\in\N$ with $0\leq b_j\leq l-j,\,\forall\,1\leq j \leq m$, \begin{equation}\label{step2}
y^{b_1}_{1}y^{b_2}_{2} \cdots y^{b_{m-1}}_{m-1}y^{\omega_m}_{m}=\sum_{\substack{c_1,\cdots,c_m\in\N\\
0\leq c_i\leq \ell-i, \forall\,1\leq i\leq m}}r_{c_1,\cdots,c_m}y^{c_1}_{1}y^{c_2}_{2} \cdots y^{c_{m-1}}_{m-1}y^{c_m}_{m}, \end{equation}
where $r_{c_1,\cdots,c_m}\in K$ for each $m$-tuple $(c_1,\cdots,c_m)$.

We use induction on $m$. If $m=1$, there is nothing to prove as $y_1^{\omega_1}=0$ whenever $\omega_1>\ell-1$. Suppose that (\ref{step2}) holds for any $1\leq k\leq m$.

We now consider the case where $k=m+1$. Applying Lemma \ref{HLlem0} for $s=m+1$, we get that
$$
y^{\ell-m}_{m+1}=-\sum_{\substack{l_1,\cdots,l_{m+1}\in\N\\ l_{m+1}\neq \ell-m, l_1+\cdots +l_{m+1}=\ell-m}} y^{l_1}_{1}y^{l_2}_{2}\cdots y^{l_{m+1}}_{m+1}.
$$

It follows that $$
\begin{aligned}
&\quad\,y^{b_1}_{1}y^{b_2}_{2} \cdots y^{b_{m-1}}_{m-1}y^{b_m}_{m}y^{\omega_{m+1}}_{m+1}  \\
&=y^{b_1}_{1}y^{b_2}_{2} \cdots y^{b_{m-1}}_{m-1}y^{b_m}_{m}y^{\omega_{m+1}-(\ell-m)}_{m+1}y^{\ell-m}_{m+1}  \\
&=-\sum_{\substack{l_1,\cdots,l_{m+1}\in\N\\ l_{m+1}\neq \ell-m, l_1+\cdots +l_{m+1}=\ell-m}}y^{b_1+l_1}_{1}y^{b_2+l_2}_{2} \cdots y^{b_{m-1}+l_{m-1}}_{m-1}y^{b_m+l_m}_{m}y^{b^{\prime}_{m+1}}_{m+1},
\end{aligned}
$$
where $b^{\prime}_{m+1}:=\omega_{m+1}-(l-m)+l_{m+1}$.

Our purpose is to show that \begin{equation}\label{step2a}
y^{b_1}_{1}y^{b_2}_{2} \cdots y^{b_{m}}_{m}y^{\omega_{m+1}}_{m+1}\in \text{$K$-Span}\bigl\{y^{c_1}_{1}y^{c_2}_{2} \cdots y^{c_{m}}_{m}y^{c_{m+1}}_{m+1}\bigm|
c_i\in\N, 0\leq c_i\leq \ell-i, \forall\,1\leq i\leq m+1\bigr\}. \end{equation}

We use induction on $\omega_{m+1}$. Suppose that for any $b_1,\cdots,b_m\in\N$ and any $0\leq b<\omega_{m+1}$, we have that $$
y^{b_1}_{1}y^{b_2}_{2} \cdots y^{b_{m}}_{m}y^{b}_{m+1}\in \text{$K$-Span}\bigl\{y^{c_1}_{1}y^{c_2}_{2} \cdots y^{c_{m}}_{m}y^{c_{m+1}}_{m+1}\bigm|
c_i\in\N, 0\leq c_i\leq \ell-i, \forall\,1\leq i\leq m+1\bigr\}.
$$
We are now going to prove (\ref{step2a}). If $b^{\prime}_{m+1}\leq l-m$, then by induction hypothesis we have that $$
y^{b_1+l_1}_{1}y^{b_2+l_2}_{2} \cdots y^{b_{m-1}+l_{m-1}}_{m-1}y^{b_m+l_m}_{m}\in \text{$K$-Span}\bigl\{y^{c_1}_{1}y^{c_2}_{2} \cdots y^{c_{m-1}}_{m-1}y^{c_{m}}_{m}\bigm|c_1,\cdots,c_m\in\N, 0\leq c_i\leq \ell-i, \forall\,1\leq i\leq m\bigr\}, $$
hence $$\begin{aligned}
&y^{b_1+l_1}_{1}y^{b_2+l_2}_{2} \cdots y^{b_{m-1}+l_{m-1}}_{m-1}y^{b_m+l_m}y^{b^{\prime}_{m+1}}_{m+1}\in \\
&\qquad\text{$K$-Span}\bigl\{y^{c_1}_{1}y^{c_2}_{2} \cdots y^{c_{m}}_{m}y^{c_{m+1}}_{m+1}\bigm|c_1,\cdots,c_{m+1}\in\N, 0\leq c_i\leq \ell-i, \forall\,1\leq i\leq m+1\bigr\}. \end{aligned}$$

Therefore, it remains to consider those terms which satisfy that $b^{\prime}_{m+1}>\ell-m$. Since $l_1+\cdots +l_{m+1}=\ell-m$ and $l_{m+1}\neq \ell-m$, we have
$0\leq l_{m+1}\leq \ell-m-1$, furthermore, we have $b^{\prime}_{m+1}\leq\omega_{m+1}-1$. By our induction hypothesis on $\omega_{m+1}$, we have that
$$
y^{b_1+l_1}_{1}y^{b_2+l_2}_{2} \cdots y^{b_{m-1}+l_{m-1}}_{m-1}y^{b_m+l_m}_{m}y^{b^{\prime}_{m+1}}_{m+1}\in\text{$K$-Span}\bigl\{y^{c_1}_{1}y^{c_2}_{2} \cdots y^{c_{m}}_{m}y^{c_{m+1}}_{m+1}\bigm|
c_i\in\N, 0\leq c_i\leq \ell-i, \forall\,1\leq i\leq m+1\bigr\}.
$$
Therefore, we can conclude that (\ref{step2a}) always holds. This completes the proof of (\ref{step2}).

Now we have proved that the elements in (\ref{AKbasis}) form a $K$-linear generator of $\HH_{\ell,n}^{(0)}$. Since the set (\ref{AKbasis}) has cardinality equal to $\bigl(^\ell_n\!\bigr)(n!)^2$, which is equal to the dimension of $\HH_{\ell,n}^{(0)}$, it follows that the elements in (\ref{AKbasis}) must form a $K$-basis of $\HH_{\ell,n}^{(0)}$. This completes the proof of the theorem.
\end{proof}

\begin{rem} We shall call the basis (\ref{AKbasis}) a \textbf{monomial basis} of $\HH_{\ell,n}^{(0)}$. It bears much resemblance to the Ariki--Koike basis of the cyclotomic Hecke algebra of type $G(\ell,1,n)$. For arbitrary cyclotomic quiver Hecke algebras the Question \ref{QA} on how to construct a monomial basis remains open. Anyhow, we regard Theorem \ref{mainthm1} as a first step in our effort of answering that open question.
\end{rem}

\bigskip
\section{A basis of the center}

The purpose of this section is to give an explicit basis of the center of $\HH_{\ell,n}^{(0)}$. Let $Z:=Z(\HH_{\ell,n}^{(0)})$ be the center of $\HH_{\ell,n}^{(0)}$.

\begin{dfn} For each $\bmu\in\P_0$, we define $$
b_{\bmu}:=\psi_{w_0}y_{\bmu}\psi_{w_0}y_1^{n-1}y_2^{n-2}\cdots y_{n-1} .	
$$
\end{dfn}

By Definition \ref{wf}, Corollary \ref{fprimecor}, Lemma \ref{1relation} and Lemma \ref{2usefulProperties}, we have that $$\begin{aligned}
\wf_{1,1}&=F'_{1,1}=\psi_{w_0}^\ast y_1^{n-1}y_2^{n-2}\cdots y_{n-1}\psi_{w_0}y_1^{n-1}y_2^{n-2}\cdots y_{n-1}\\
&=(\psi_{w_0}y_1^{n-1}y_2^{n-2}\cdots y_{n-1})\psi_{w_0}y_1^{n-1}y_2^{n-2}\cdots y_{n-1}\\
&=(-1)^{n(n-1)/2}\psi_{w_0}y_1^{n-1}y_2^{n-2}\cdots y_{n-1}=F_{1,1} .\end{aligned}
$$

Note that each $y_\bmu$ has a left factor $y_1^{n-1}y_2^{n-2}\cdots y_{n-1}$. It follows that $$
b_{\bmu}\in \wf_{1,1}\HH_{\ell,n}^{(0)} \wf_{1,1} \cong\End_{\HH_{\ell,n}^{(0)}}(\wf_{1,1}\HH_{\ell,n}^{(0)})\cong\End_{\HH_{\ell,n}^{(0)}}(P_0).
$$
Suppose further that $\theta(\bmu)=(k_1,\cdots,k_n)$, where $1\leq k_1<k_2<\cdots<k_n\leq\ell$. Then by (\ref{dominanceBigger}), $$\begin{aligned}
b_{\bmu}&=\psi_{w_0}y_1^{\ell-k_1}y_2^{\ell-k_2}\cdots y_{n}^{\ell-k_n}\psi_{w_0}y_1^{n-1}y_2^{n-2}\cdots y_{n-1}\\
&\equiv(-1)^{n(n-1)/2}\psi_{w_0}y_1^{\ell-k_1}y_2^{\ell-k_2}\cdots y_{n}^{\ell-k_n}\pmod{(\HH_{\ell,n}^{(0)})^{>\bmu}}\\
&\equiv(-1)^{n(n-1)/2}\psi_{w_0,1}^{\bmu}\pmod{(\HH_{\ell,n}^{(0)})^{>\bmu}}.
\end{aligned}
$$
It follows that $\{b_{\bmu}|\bmu\in\P_0\}$ are $K$-linearly independent elements in $\wf_{1,1}\HH_{\ell,n}^{(0)} \wf_{1,1}$.

\begin{lem} \label{dimCenter} The elements in $\{b_{\bmu}|\bmu\in\P_0\}$ form a $K$-basis of $\wf_{1,1}\HH_{\ell,n}^{(0)} \wf_{1,1}$. Moreover, the basic algebra
$\End_{\HH_{\ell,n}^{(0)}}(P_0)$ of $\HH_{\ell,n}^{(0)}$ is commutative and is isomorphic to the center $Z$ of $\HH_{\ell,n}^{(0)}$. In particular,
$\dim_K Z=\bigl(^\ell_n\!\bigr)$.
\end{lem}

\begin{proof} Since $\#\P_0=\bigl(^\ell_n\!\bigr)$ and $\wf_{1,1}\HH_{\ell,n}^{(0)} \wf_{1,1}\cong\End_{\HH_{\ell,n}^{(0)}}(P_0)$, it suffices to show that
$\dim_K\End_{\HH_{\ell,n}^{(0)}}(P_0)=\bigl(^\ell_n\!\bigr)$. By Lemma \ref{bili} and Corollary \ref{gradedcharacters}, we know that $[P_0:D_0]=\bigl(^\ell_n\!\bigr)$ and hence $\dim_K\End_{\HH_{\ell,n}^{(0)}}(P_0)=\bigl(^\ell_n\!\bigr)$ as required. Thus the first part of the lemma follows from this together with the discussion in the paragraph above this lemma.

It remains to show  that the endomorphism algebra $\End_{\HH_{\ell,n}^{(0)}}(P_0)$ is commutative. Once this is proved, then as $\HH_{\ell,n}^{(0)}$ is Morita equivalent to $\End_{\HH_{\ell,n}^{(0)}}(P_0)$. It follows from \cite[(3.54)(iv)]{CuR} that $$
Z=Z(\HH_{\ell,n}^{(0)})\cong Z\bigl(\End_{\HH_{\ell,n}^{(0)}}(P_0)\bigr)=\End_{\HH_{\ell,n}^{(0)}}(P_0),
$$
as required.

In order to show that $\End_{\HH_{\ell,n}^{(0)}}(P_0)$ of $\HH_{\ell,n}^{(0)}$ is commutative, it suffices to show that $\wf_{1,1}\HH_{\ell,n}^{(0)} \wf_{1,1}$
is commutative. Furthermore, it is enough to show that $b_\bmu b_\bnu=b_\bnu b_\bmu$ for any $\bmu,\bnu\in\P_0$.

By definition, $$\begin{aligned}
&\quad\,b_\bmu b_\bnu=\psi_{w_0}y_{\bmu}\psi_{w_0}y_1^{n-1}y_2^{n-2}\cdots y_{n-1}\psi_{w_0}y_{\bnu}\psi_{w_0}y_1^{n-1}y_2^{n-2}\cdots y_{n-1}\\
&=(-1)^{n(n-1)/2}\psi_{w_0}\bigl(y_{\bmu}\psi_{w_0}y_{\bnu}\bigr)\psi_{w_0}y_1^{n-1}y_2^{n-2}\cdots y_{n-1} .
\end{aligned}
$$
We set $J_{1,1}:=\sum_{j=1}^{n-1}\psi_j\HH_{\ell,n}^{(0)}+\sum_{j=1}^{n-1}\HH_{\ell,n}^{(0)}\psi_j$. Using the graded cellular basis $\{\psi_{w,u}^\bmu|\bmu\in\P_0\}$ of $\HH_{\ell,n}^{(0)}$, we can write $$
y_{\bmu}\psi_{w_0}y_{\bnu}\equiv\sum_{\brho\in\P_0}c_{\brho}y_{\brho}\pmod{J_{1,1}} ,
$$
where $c_{\alpha}\in K$ for each $\alpha\in\P_0$. Applying the anti-involution ``$\ast$" on both sides of the above equality, we get that $$
y_{\bnu}\psi_{w_0}y_{\bmu}\equiv\sum_{\brho\in\P_0}c_{\brho}y_{\brho}\pmod{J_{1,1}} .
$$
Now using Lemma \ref{2usefulProperties} we can deduce  that $$
b_\bmu b_\bnu=(-1)^{n(n-1)/2}\sum_{\brho\in\P_0}c_{\brho}\psi_{w_0}y_{\brho}\psi_{w_0}y_1^{n-1}y_2^{n-2}\cdots y_{n-1}=b_\bnu b_\bmu,
$$
as required. This completes the proof of the lemma.
\end{proof}

\begin{dfn}\label{centerBasis} Let $\bmu\in\P_0$ with $\theta(\bmu)=(k_1,\cdots,k_n)$, where $1\leq k_1<k_2<\cdots<k_n\leq \ell$. Inside the quiver Hecke algebra $\HH_{n}^{(0)}$, we define $z(\bmu)\in K[y_1,\cdots,y_n]$ such that $$
y_1^{\ell-k_1}\cdots y_n^{\ell-k_n}\psi_{w_0}=z(\bmu)+\sum_{r=1}^{n-1}\psi_r h_r ,
$$
where $h_r\in\HH_n^{(0)}$ for each $1\leq r<n$. We define $$
z_{\bmu}:=\pi(z(\bmu))\in\HH_{\ell,n}^{(0)} .
$$
\end{dfn}

It is clear that $z_\bmu$ is a homogeneous element with degree $2\ell n-n(n-1)-2\sum_{i=1}^{n}k_i$.

\begin{lem} \label{center1} Let $\bmu\in\P_0$. Then $z(\bmu)$ is a symmetric polynomial in $y_1,\cdots,y_n$. In particular, $z(\bmu)$ lives inside the center of $\HH_n^{(0)}$ and hence $z_{\bmu}$ lives inside the center of $\HH_{\ell,n}^{(0)}$. Moreover,  $z(\blam_{\max})=(-1)^{n(n-1)/2}(y_1\cdots y_n)^{\ell-n}$ and $z(\blam_{\min})=(-1)^{n(n-1)/2}$.
\end{lem}

\begin{proof} It suffices to show that $z(\bmu)$ is symmetric in $y_r,y_{r+1}$ for each $1\leq r<n-1$. In fact, for any $1\leq r<n-1$ and $a,b\in\N$, if $a>b$ then $$
y_r^ay_{r+1}^b\psi_r=y_r^{a-b}(y_ry_{r+1})^b\psi_r=y_r^{a-b}\psi_r(y_ry_{r+1})^b\equiv -\Bigl(\sum_{k=0}^{a-b-1}y_{r}^{k}y_{r+1}^{a+b-1-k}\Bigr)(y_ry_{r+1})^b\pmod{\sum_{r=1}^{n-1}\psi_r\HH_n^{(0)}} ;
$$
if $a<b$, then $$
y_r^ay_{r+1}^b\psi_r=y_{r+1}^{b-a}(y_ry_{r+1})^a\psi_r=y_{r+1}^{b-a}\psi_r(y_ry_{r+1})^a\equiv \Bigl(\sum_{k=0}^{b-a-1}y_{r}^{k}y_{r+1}^{b-a-1-k}\Bigr)(y_ry_{r+1})^a\pmod{\sum_{r=1}^{n-1}\psi_r\HH_n^{(0)}} ;
$$
if $a=b$, then $y_r^ay_{r+1}^b\psi_r=(y_ry_{r+1})^a\psi_r=\psi_r(y_ry_{r+1})^a\in\sum_{r=1}^{n-1}\psi_r\HH_n^{(0)}$ . This implies that for any monomial
$y_1^{c_1}\cdots y_n^{c_n}\in\HH_n^{(0)}$, $$
y_1^{c_1}\cdots y_n^{c_n}\psi_r\equiv f_r(y_1,\cdots,y_n)\pmod{\sum_{r=1}^{n-1}\psi_r\HH_n^{(0)}},
$$
where $f_r(y_1,\cdots,y_n)\in K[y_1,\cdots,y_n]$ is symmetric in $y_r,y_{r+1}$.

Since for each $1\leq r<n$, $w_0$ has a reduced expression which ends with $s_r$ and the element $z(\bmu)$ is uniquely determined by $\bmu$ by Lemma \ref{nilHeckeCenter}, it follows that $z(\bmu)$ is symmetric in $y_r, y_{r+1}$ for any $1\leq r<n-1$. Hence $z(\bmu)$ is symmetric in $y_1,\cdots,y_n$. This completes the proof of the first part of the lemma. The second part of the lemma follows from Lemma \ref{1relation} and direct calculation.
\end{proof}

\begin{lem} \label{keylem2} 1) For each $\bmu\in\P_0$, we have  that $$
\psi_{w_0}y_\bmu\psi_{w_0}y_1^{n-1}\cdots y_{n-1}=\psi_{w_0}y_1^{n-1}\cdots y_{n-1}z_\bmu . $$ In particular, $$
\psi_{w_0}y_\bmu\equiv (-1)^{n(n-1)/2}\psi_{w_0}y_1^{n-1}y_2^{n-2}\cdots y_{n-1}z_\bmu\pmod{(\HH_{\ell,n}^{(0)})^{>\bmu}}.
$$

2) As a left $Z$-module, $P_0\cong Z^{\oplus n!}$. In particular, $P_0$ is a free $Z$-module of rank $n!$.
\end{lem}

\begin{proof} First, since $\HH_{\ell,n}^{(0)}\cong P_0^{\oplus n!}$, it follows that the center $Z$ must act faithfully on $P_0$. In other words, the left multiplication defines an injective homomorphism $\iota: Z\hookrightarrow\End_{\HH_{\ell,n}^{(0)}}(P_0)$. Comparing the dimensions of both sides, we can deduce that $\iota$ is an isomorphism. On the other hand, by Lemma \ref{dimCenter}, $$
0\neq b_\bmu\in\wf_{1,1}\HH_{\ell,n}^{(0)} \wf_{1,1}\cong\End_{\HH_{\ell,n}^{(0)}}(P_0) .
$$
It follows that there exists a unique nonzero homogeneous element $z'_\bmu$ with degree $2(\ell-k_1+\cdots+\ell-k_n)-(n-1)n$ such that \begin{equation}\label{keyprop2}
\psi_{w_0}y_\bmu\psi_{w_0}y_1^{n-1}\cdots y_{n-1}=z'_\bmu\psi_{w_0}y_1^{n-1}\cdots y_{n-1}
=\psi_{w_0}z'_\bmu y_1^{n-1}\cdots y_{n-1}=\psi_{w_0}y_1^{n-1}\cdots y_{n-1}z'_\bmu.
\end{equation}
By Lemma \ref{center1} and Lemma \ref{2usefulProperties}, we can see that $z'_\bmu=z_\bmu$. In particular, $z_\bmu\neq 0$.

Since $$\begin{aligned}
\psi_{w_0}y_\bmu\psi_{w_0}y_1^{n-1}\cdots y_{n-1}
&\equiv(-1)^{n(n-1)/2}\psi_{w_0}y_\bmu\equiv(-1)^{n(n-1)/2}\psi_{w_0,1}^{\bmu}\pmod{(\HH_{\ell,n}^{(0)})^{>\bmu}} ,
\end{aligned}
$$
it follows that $\psi_{w_0}y_\bmu\equiv (-1)^{n(n-1)/2}\psi_{w_0}y_1^{n-1}y_2^{n-2}\cdots y_{n-1}z_\bmu\pmod{(\HH_{\ell,n}^{(0)})^{>\bmu}}$. This proves 1).

Recall that $\wf_{1,1}=F'_{1,1}=(-1)^{n(n-1)/2}\psi_{w_0}y_1^{n-1}y_2^{n-2}\cdots y_{n-1}$. It follows from 1) that for any $\bmu\in\P_0$ and $w\in\Sym_n$, $$
\wf_{1,1}z_{\bmu}\psi_w\equiv\psi_{w_0,w}^{\bmu}\pmod{(\HH_{\ell,n}^{(0)})^{>\bmu}} .
$$
In particular, the elements in the set $\{\wf_{1,1}z_{\bmu}\psi_w|\bmu\in\P_0,w\in\Sym_n\}$ must be $K$-linearly independent. Since it has the cardinality
$\bigl(^\ell_n\!\bigr)n!$, we can deduce that it is a $K$-basis of the right $\HH_{\ell,n}^{(0)}$-module $P_0\cong\wf_{1,1}\HH_{\ell,n}^{(0)}$. Since $P_0$ is a faithful $Z$-module, it follows that for any $z\in Z$, $\wf_{1,1}z=0$ if and only if $z=0$. For each $w\in\Sym_n$, the subspace spanned by the basis elements in $\{\wf_{1,1}z_{\bmu}\psi_w|\bmu\in\P_0\}$ is an $Z$-submodule of $P_0$ which is isomorphic to $Z$. This proves that $P$ is a free $Z$-module with rank $n!$.
\end{proof}


\begin{thm} \label{mainthm2} The elements in the set $\{z_{\bmu}|\bmu\in\P_0\}$ form a $K$-basis of the center $Z:=Z(\HH_{\ell,n}^{(0)})$ of $\HH_{\ell,n}^{(0)}$. In particular, the center of $\HH_{\ell,n}^{(0)}$ is the set of symmetric polynomials in $y_1,\cdots,y_n$.
\end{thm}

\begin{proof} Since the elements in $\{b_\bmu|\bmu\in\P_0\}$ are $K$-linearly independent, it follows that the elements in $\{z_{\bmu}|\bmu\in\P_0\}$ are $K$-linearly independent and hence form a $K$-basis of the center $Z:=Z(\HH_{\ell,n}^{(0)})$ by dimension consideration. By Lemma \ref{center1}, each $z_\bmu$ is a symmetric polynomial in $y_1,\cdots,y_n$, hence the center of $\HH_{\ell,n}^{(0)}$ is the set of symmetric polynomials in $y_1,\cdots,y_n$. This completes the proof of the theorem.
\end{proof}

The following proposition gives a generalization of Corollary \ref{matrixUnit1}. It can be regarded as a cyclotomic analogue of the results in \cite[Proposition 3.5]{Lau} and \cite[Theorem 4.5]{KLM}.

\begin{prop} \label{matrixIso} Let $\{E_{i,j}|1\leq i,j\leq n!\}$ be the matrix units of the full matrix algebra $M_{n!\times n!}(K)$. Then the map $$
E_{i,j}\otimes z\mapsto \wf_{w_i,w_j}z,\quad\forall\, 1\leq i,j\leq n!, z\in Z,
$$
extends linearly to a well-defined $K$-algebra isomorphism $\eta$ from $M_{n!\times n!}(K)\otimes _{K}Z$ onto $\HH_{\ell,n}^{(0)}$. In particular, $\HH_{\ell,n}^{(0)}\cong M_{n!\times n!}(Z)$.
\end{prop}

\begin{proof} In view of Theorem \ref{mainthm0}, it is clear that $\eta$ is a well-defined $K$-algebra homomorphism. By Lemma \ref{dimCenter}, it suffices to show that $\eta$ is an injective map.

Suppose that $\eta(x)=0$, where $x=\sum_{1\leq i,j\leq n!}E_{i,j}z_{i,j}$, where $z_{i,j}\in Z$ for each pair $(i,j)$.
Then $$
\sum_{1\leq i,j\leq n!}\wf_{w_i,w_j}z_{i,j}=\eta(x)=0 .
$$
For any pair $(i,j)$ with $1\leq i,j\leq n!$, left multiplying with $\wf_{w_j,w_i}$ and right multiplying with $\wf_{w_j,w_j}$ we get (by Theorem \ref{mainthm0}) that $$
\wf_{w_j,w_j}z_{i,j}=\sum_{1\leq k,l\leq n!}(\wf_{w_j,w_i}\wf_{w_k,w_l}\wf_{w_j,w_j})z_{k,l}\Bigr)=\wf_{w_j,w_i}\Bigl(\sum_{1\leq k,l\leq n!}\wf_{w_k,w_l}z_{k,l}\Bigr)\wf_{w_i,w_j}=0 .
$$
Since $\wf_{w_j,w_j}\HH_{\ell,n}^{(0)}\cong P_0$ as ungraded right $\HH_{\ell,n}^{(0)}$-module and $Z$ acts faithfully on $P_0$, it follows that $z_{i,j}=0$.
This proves that $x=0$ and hence $\eta$ is injective. Finally, comparing the dimensions of both sides, we see that $\eta$ is an isomorphism.
\end{proof}

\bigskip

\section{A homogeneous symmetrizing form on $\HH_{\ell,n}^{(0)}$}

By the work \cite{SVV} of Shan, Varagnolo and Vasserot, each cyclotomic quiver Hecke algebra can be endowed with a homogeneous symmetrizing form which makes it into a graded symmetric algebra (see Remark \ref{FinalRem} and \cite[Section 6.3]{HuMathas:GradedCellular} for the type $A$ case). In particular,
the nilHecke algebra $\HH_{\ell,n}^{(0)}$ is a graded symmetric algebra. However,  the SVV symmetrizing form $\Tr^{\text{SVV}}$ is defined in an inductive manner which relies on some deep results about certain decompositions of the cyclotomic quiver Hecke algebras which come from the biadjointness of the $i$-induction functors and $i$-restriction functors in the work \cite{KK} of Kang and Kashiwara and \cite{K} of Kashiwara. It is rather difficult to compute the explicit value of the form $\Tr^{\text{SVV}}$ on any specified homogeneous element in the cyclotomic quiver Hecke algebra  because its inductive definition involves some mysterious correspondence (i.e., $z\mapsto\widetilde{z}, \ell\mapsto\widetilde{\pi}_\ell$ in \cite[Theorem 3.8]{SVV}) whose explicit descriptions are not available. In this section, we shall introduce a new homogeneous symmetrizing form $\Tr$ such that the value of the form $\Tr$ on each graded cellular basis element of $\HH_{\ell,n}^{(0)}$ is explicitly given. We will prove in the next section that this form $\Tr$ actually coincides with Shan-Varagnolo-Vasserot's symmetrizing form $\Tr^{\text{SVV}}$ on $\HH_{\ell,n}^{(0)}$.

The following result seems to be well-known. We add a proof as we can not find a suitable reference.

\begin{lem} \label{SymMorita} Let $A, B$ be two finite dimensional (ungraded) $K$-algebras. Suppose that $B$ is Morita equivalent to $A$. Then there exists a $K$-linear map $\rho: A^\ast\rightarrow B^\ast$ such that for any symmetrizing form $\tau\in A^\ast$ on $A$, $\rho(\tau)\in B^\ast$ is a symmetrizing form on $B$. In particular, if $A$ is a symmetric algebra over $K$, then $B$ is a symmetric algebra over $K$ too.\end{lem}

\begin{proof} By assumption, $B^{\text{op}}\cong\End_A(P)$ for a finite dimensional (ungraded) projective left $A$-module $P$. Moreover, there exists a natural number $k$
such that $A^{\oplus k}\cong P\oplus P'$ as left $A$-modules. Let $e$ be the idempotent of $M_{k\times k}(A)$ which corresponds to the map $A^{\oplus k}\overset{\text{pr}}{\twoheadrightarrow} P\overset{\text{$\iota$}}{\hookrightarrow}A^{\oplus k}$. Then we have that
$B^{\text{op}}\cong\End_A\bigl(P\bigr)\cong eM_{k\times k}(A)e$.

We define $\rho_0: A^\ast\rightarrow\bigl(M_{k\times k}(A)\bigr)^\ast$ as follows: for any $f\in A^\ast$ and $(a_{i,j})_{k\times k}\in M_{k\times k}(A)$, $$
\rho_0(f)\Bigl((a_{i,j})_{k\times k}\Bigr):=f\bigl(\sum_{i=1}^{k}a_{ii}\bigr) .
$$
We also define $\text{res}: \bigl(M_{k\times k}(A)\bigr)^\ast\rightarrow\bigl(eM_{k\times k}(A)e\bigr)^\ast$ as follows: for any $f\in \bigl(M_{k\times k}(A)\bigr)^\ast$ and
$(a_{i,j})_{k\times k}\in M_{k\times k}(A)$, $$
\text{res}(f)\Bigl(e(a_{i,j})_{k\times k}e\Bigr):=f\Bigl(e(a_{i,j})_{k\times k}e\Bigr).
$$
It is easy to check that $\rho:=\text{res}\circ\rho_0$ has the property that  for any symmetrizing form $\tau\in A^\ast$ on $A$, $\rho(\tau)\in B^\ast$ is a symmetrizing form on $\End_A\bigl(P\bigr)\cong eM_{k\times k}(A)e\cong B^{\text{op}}$. It is clear that $\rho(\tau)$ is a symmetrizing form on $B$ too. This completes the proof of the lemma.
\end{proof}

\begin{lem}\label{homogeneous1} Let $A=\oplus_{k=0}^{m}A_k$ be a finite dimensional positively $\Z$-graded $K$-algebra. Let $\tau$ be a (not necessarily homogeneous) symmetrizing form on $A$. We define $\widetilde{\tau}: A^\ast\rightarrow K$ as follows: for any homogeneous element $y\in A$, $$
\widetilde{\tau}(y):=\begin{cases} \tau(x), &\text{if $\deg x=m$;}\\
0, &\text{otherwise.}
\end{cases}
$$
Then $\widetilde{\tau}$ can be linearly extended to a well-defined homogeneous symmetrizing form on $A$.
\end{lem}

\begin{proof} This is clear.
\end{proof}

The following definition comes from \cite[3.1.5]{SVV}.

\begin{dfn} We define $$
d_\Lam:=2\ell n-2n^2 .
$$
\end{dfn}

Recall that by Theorem \ref{mainthm2} the center $Z$ is a positively $\Z$-graded $K$-algebra with each homogeneous component being one dimensional.
In particular, $\deg z\leq d_\Lam$ for all $z\in Z$, and $\deg z_{\blam_{\maxim}}=d_\Lam$.

\begin{lem} \label{Zsymmetric} The center $Z$ can be endowed with a homogeneous symmetrizing form of degree $-d_\Lam$ as follows: for any homogeneous element $z\in Z$, $$
\tr(z):=\begin{cases} 1, &\text{if $z=z_{\blam_\maxim}$;}\\
0, &\text{if $\deg z<d_\Lam$.}
\end{cases} .
$$
In particular, $Z$ is a graded symmetric algebra over $K$.
\end{lem}

\begin{proof} By Lemma \ref{dimCenter}, we know that $Z$ is Morita equivalent to $\HH_{\ell,n}^{(0)}$. Since $\HH_{\ell,n}^{(0)}$ is a symmetric algebra by \cite{SVV}, we can deduce from Lemma \ref{SymMorita} and Lemma \ref{homogeneous1} that $Z$ is a graded symmetric algebra too.

On the other hand, by Lemma \ref{dimCenter} and Corollary \ref{gradedCartan}, we know that the center $Z$ is positively graded $K$-algebra with each homogeneous component being one dimensional. Therefore, we are in a position to apply \cite[Proposition 3.9]{Hlam} or Lemma \ref{SymMorita} and Lemma \ref{homogeneous1} to show that $\tr$ is a well-defined homogeneous symmetrizing form on $Z$. This completes the proof of the lemma.
\end{proof}

Since $\tr$ is a homogeneous symmetrizing form on $Z$, for each nonzero homogeneous element $0\neq z\in Z$, there exists a homogeneous element $\hat{z}\in Z$ with degree $d_\Lam-\deg z$ such that $\tr(z\hat{z})\neq 0$. This motivates the following definition.

\begin{dfn} \label{zhat} For each $\blam\in\P_0$, we fix a nonzero homogeneous element $\widehat{z}_\blam\in Z$ with degree $d_\Lam-\deg z_{\blam}$ such that $\tr(z_\blam\widehat{z}_\blam)\neq 0$.
\end{dfn}

Now we are using Proposition \ref{matrixIso} and Lemma \ref{Zsymmetric} to define a homogeneous symmetrizing form $\hat{\Tr}$ on $\HH_{\ell,n}^{(0)}$ as follows: for any $1\leq i,j\leq n!$ and any homogeneous element $z\in Z$,
$$
\hat{\Tr}\bigl(\wf_{w_i,w_j}z\bigr):=\begin{cases} c, &\text{if $i=j$ and $z=cz_{\blam_{\maxim}}$ for some $c\in K$;}\\
0, &\text{if $i\neq j$ or $\deg z<d_\Lam$.}
\end{cases}
$$

\begin{lem} \label{1stForm} The map $\hat{\Tr}$ extends linearly to a well-defined homogeneous symmetrizing form of degree $-d_\Lam$ on $\HH_{\ell,n}^{(0)}$.
\end{lem}

\begin{proof} This follows directly from Lemma \ref{Zsymmetric} and Proposition \ref{matrixIso}.
\end{proof}

\begin{rem} \label{FinalRem} In \cite{SVV}, Shan, Varagnolo and Vasserot show that each cyclotomic quiver Hecke algebra $\R[\beta]$ can be endowed with a homogeneous symmetrizing form $\Tr^{\text{SVV}}$ of degree $d_{\Lam,\beta}$ which makes it into a graded symmetric algebra, where $$
\beta\in Q_n^+,\quad \Lam\in P^+,\quad d_{\Lam,\beta}:=2(\Lam,\beta)-(\beta,\beta).
$$
In the type $A$ case we consider the cyclic quiver or linear quiver with vertices labelled by $\Z/e\Z$, where $e\neq 1$ is a non-negative integer. In this case, $\R[\beta]$ can be identified with the block of the cyclotomic Hecke algebra of type $A$ which corresponds to $\beta$ by Brundan-Kleshchev's isomorphism (\cite{BK:GradedKL}) when the ground field $K$ contains a primitive $e$th root of unity or $e$ is equal to the characteristic of the ground field $K$. There is another homogeneous symmetrizing form $\Tr^{\text{HM}}$ which can be defined (cf. \cite[Section 6.3]{HuMathas:GradedCellular}) as follows: let $\tau$ be the ungraded symmetrizing form on $\R[\beta]$ defined in \cite{MM} (non-degenerate case) and \cite{BK:HigherSchurWeyl} (degenerate case).
Following \cite[Definition 6.15]{HuMathas:GradedCellular}, for any homogeneous element $x\in\R[\beta]$, we define $$
\Tr^{\text{HM}}(x):=\begin{cases} \tau(x), &\text{if $\deg(x)=d_{\Lam,\beta}$;}\\
0, &\text{otherwise.}
\end{cases} .
$$
By the proof of \cite[Theorem 6.17]{HuMathas:GradedCellular}, $\Tr^{\text{HM}}$ is a homogenous symmetrizing form on $\R[\beta]$ of degree $-d_{\Lam,\beta}$.
The associated homogenous bilinear form $\<-,-\>$ on $\R[\beta]$ of degree $-d_{\Lam,\beta}$ can be defined as follows: $\<x,y\>:=\Tr^{\text{HM}}(xy)$. We take this chance to remark that the bilinear form $\<-,-\>_\beta$ in the paragraph above \cite[Theorem 6.17]{HuMathas:GradedCellular} should be replaced with the bilinear form $\<-,-\>$ we defined here.
\end{rem}

\begin{conj} The two symmetrizing forms $\Tr^{\text{SVV}}$ and $\Tr^{\text{HM}}$ on $\R[\beta]$ differ by a nonzero scalar in $K$.
\end{conj}

\begin{dfn} For each $\bmu\in\P_0$ and $z_1,z_2\in\Sym_n$, we define $$
\phi_{z_1,z_2}^\bmu:=\psi_{z_1}^\ast y_1^{n-1}y_2^{n-2}\cdots y_{n-1}\psi_{w_0}y_\bmu\psi_{w_0}y_1^{n-1}y_2^{n-2}\cdots y_{n-1}\psi_{z_2} .
$$
\end{dfn}

\begin{lem} \label{phibasis} 1) For each $\bmu\in\P_0$ and $z_1,z_2\in\Sym_n$, we have that $$
\phi_{w_0z_1,z_2}^\bmu=F'_{z_1,z_2}z_\bmu=\psi_{w_0z_1}^\ast y_1^{n-1}y_2^{n-2}\cdots y_{n-1}\psi_{w_0}y_1^{n-1}y_2^{n-2}\cdots y_{n-1}z_{\bmu}\psi_{z_2}
$$ and $$
\phi_{z_1,z_2}^\bmu\equiv\psi_{z_1,z_2}^\bmu\pmod{(\HH_{\ell,n}^{(0)})^{>\bmu}} .
$$

2) The elements in the set $\{\phi_{z_1,z_2}^\bmu|\bmu\in\P_0,z_1,z_2\in\Sym_n\}$ form a homogeneous $K$-basis of $\HH_{\ell,n}^{(0)}$.
\end{lem}

\begin{proof} The first part of 1) follows from Lemma \ref{keylem2}, while the second part of 1) follows from Lemma \ref{1relation}.
Finally, 2) follows from 1) and (\ref{HMcellular}).
\end{proof}

We are going to define another homogeneous symmetrizing form ``$\Tr$" on $\HH_{\ell,n}^{(0)}$. Let $\blam\in\P_0$ and $w,u\in\Sym_n$. By the same argument used in the proof of Lemma \ref{center1}, there is an element $z_{w,u}$ in the center $Z(\HH_{\ell,n}^{(0)})$ of $\HH_{\ell,n}^{(0)}$ such that $$
\psi_{w_0}y_1^{n-1}y_2^{n-2}\cdots y_{n-1}\psi_u\psi_{w^{-1}w_0}y_1^{n-1}y_2^{n-2}\cdots y_{n-1}\psi_{w_0}=\psi_{w_0}z_{w,u} .
$$
If $\deg z_{\blam}+\deg z_{w,u}=d_\Lam$, then we denote $c_{w,u}\in K$ the unique scalar which satisfies that $z_{w,u}z_{\blam}=c_{w,u}z_{\blam_{\max}}$.
Note that $\deg z_{\blam}+\deg z_{w,u}=d_\Lam$ if and only if $\deg\phi_{w_0w,u}^\blam=d_\Lam$.

\begin{dfn} \label{DefTr} For any $\bmu\in\P_0$ and $w,u\in\Sym_n$, we define $$
\Tr(F'_{w,u}z_\bmu)=\Tr(\phi_{w_0w,u}^\bmu):=\begin{cases} c_{w,u},&\text{if $\deg F'_{w,u}z_\bmu=d_\Lam$;}\\
0, &\text{otherwise.}
\end{cases}
$$
\end{dfn}

In particular, if $w=u$ and $\bmu=\lam_{\maxim}$ then $\Tr(\phi_{w,u}^\bmu)=1$. Note that $$\begin{aligned}
1&=\Tr(\phi_{w_0,1}^{\blam_{\max}})=\Tr(F'_{1,1}z_{\blam_{\max}})=\Tr(\psi_{w_0}^\ast y_1^{n-1}y_2^{n-2}\cdots y_{n-1}\psi_{w_0}y_1^{n-1}y_2^{n-2}\cdots y_{n-1}z_{\blam_{\max}})\\
&=(-1)^{n(n-1)/2}\Tr(\psi_{w_0}^\ast y_1^{n-1}y_2^{n-2}\cdots y_{n-1}z_{\blam_{\max}})\\
&=\Tr(\psi_{w_0}^\ast y_{\blam_{\max}}),\end{aligned}$$
which implies that \begin{equation}\label{maxvalue}
\Tr(\psi_{w_0}^\ast y_{\blam_{\max}})=1.
\end{equation}

\begin{prop} \label{keyprop1} The map $\Tr$ can be linearly extended to a well-defined homogeneous symmetrizing form of degree $-d_\Lam$ on $\HH_{\ell,n}^{(0)}$.
\end{prop}

\begin{proof} By construction, it is clear that the  map $\Tr$ can be linearly extended  to a well-defined homogeneous linear map of degree $-d_\Lam$ on $\HH_{\ell,n}^{(0)}$.

We want to show that show that $\hat{\Tr}=\Tr$. Once this is proved, it is automatically that $\Tr$ is symmetric and non-degenerate.
To this end, by Lemma \ref{phibasis}, it suffices to show that $\hat{\Tr}(F'_{z_1,z_2}z_\bmu)=\Tr(F'_{z_1,z_2}z_\bmu)$ for any $\bmu\in\P_0$ and $z_1,z_2\in\Sym_n$.

Without loss of generality we can assume that $\deg(F'_{z_1,z_2}z_\bmu)=d_\Lam$. Since $\hat{\Tr}$ is a trace form and $z_\bmu$ is central, we have that $$\begin{aligned}
&\quad\,\hat{\Tr}(F'_{z_1,z_2}z_\bmu)\\
&=\hat{\Tr}(F'_{z_1,z_1}F'_{z_1,z_2}z_\bmu)\\
&=\hat{\Tr}(\psi_{w_0z_1}^\ast y_1^{n-1}y_2^{n-2}\cdots y_{n-1}\psi_{w_0}y_1^{n-1}y_2^{n-2}\cdots y_{n-1}
\psi_{z_1}\psi_{w_0z_1}^\ast y_1^{n-1}y_2^{n-2}\cdots y_{n-1}\psi_{w_0}y_1^{n-1}y_2^{n-2}\cdots y_{n-1}\psi_{z_2}z_\bmu)\\
&=\hat{\Tr}(\psi_{w_0z_1}^\ast y_1^{n-1}y_2^{n-2}\cdots y_{n-1}\psi_{w_0}y_1^{n-1}y_2^{n-2}\cdots y_{n-1}
\psi_{w_0}y_1^{n-1}y_2^{n-2}\cdots y_{n-1}\psi_{w_0}y_1^{n-1}y_2^{n-2}\cdots y_{n-1}\psi_{z_2}z_\bmu)\\
&=\hat{\Tr}(y_1^{n-1}y_2^{n-2}\cdots y_{n-1}
\psi_{w_0}y_1^{n-1}y_2^{n-2}\cdots y_{n-1}\psi_{w_0}y_1^{n-1}y_2^{n-2}\cdots y_{n-1}\psi_{z_2}\psi_{z_1^{-1}w_0}y_1^{n-1}y_2^{n-2}\cdots y_{n-1}\psi_{w_0}z_\bmu)\\
&=\hat{\Tr}(y_1^{n-1}y_2^{n-2}\cdots y_{n-1}
\psi_{w_0}y_1^{n-1}y_2^{n-2}\cdots y_{n-1}\psi_{w_0}z_{z_1,z_2}z_\bmu)\\
&=(-1)^{n(n-1)/2}\hat{\Tr}(y_1^{n-1}y_2^{n-2}\cdots y_{n-1}
\psi_{w_0}c_{z_1,z_2}z_{\blam_{\max}})\\
&=(-1)^{n(n-1)/2}c_{z_1,z_2}\hat{\Tr}(\psi_{w_0}y_1^{n-1}y_2^{n-2}\cdots y_{n-1}z_{\blam_{\max}})\\
&=c_{z_1,z_2}\hat{\Tr}(\wf_{1,1}z_{\blam_{\max}})\\
&=c_{z_1,z_2}\\
&={\Tr}(F'_{z_1,z_2}z_\bmu) .
\end{aligned}
$$

 This completes the proof of $\hat{\Tr}=\Tr$. In particular, this implies that $\Tr$ is symmetric and non-degenerate. That says,
$\Tr$ can be linearly extended to a well-defined homogeneous symmetrizing form of degree $-d_\Lam$ on $\HH_{\ell,n}^{(0)}$.
\end{proof}

\begin{prop} \label{Comparing2Traces} We have that $$
\hat{\Tr}=\Tr .
$$
\end{prop}

\begin{proof} This follows from the proof of Proposition \ref{keyprop1}.
\end{proof}

\bigskip

\section{Comparing $\Tr$ with the Shan--Varagnolo--Vasserot symmetrizing form $\Tr^{\text{SVV}}$}

In this section, we shall compare the symmetrizing form $\Tr$ with the Shan--Varagnolo--Vasserot symmetrizing form $\Tr^{\text{SVV}}$ introduced in \cite{SVV} and show that they are actually the same.

Let $A,B$ be two $K$-algebras and $i: B\rightarrow A$ is a $K$-algebra homomorphism. Let $A^B:=\{x\in A|xb=bx,\,\,\forall\,b\in B\}$ be the centralizer of $B$ in $A$. For any $f\in A^B$, we set $$
\mu_f: A\otimes_B A\rightarrow A,\quad\, a\otimes a'\mapsto afa' .
$$

Recall that $\HH_{\ell,n}^{(0)}=\RR^{\ell\Lam_0}_{n\alpha_0}$. In the notations of \cite[\S3.1.4]{SVV}, we set \begin{equation}\label{lambda0}
\lam_0:=\<\ell\Lam_0-(n-1)\alpha_0,\alpha_0^{\vee}\>=\ell-2(n-1).\end{equation}
We first recall the definition of $\Tr^{\text{SVV}}$ in the case of nilHecke algebra $\RR^{\ell\Lam_0}_{n\alpha_0}$.

\begin{dfn}\label{ptilde} \text{(\cite{KK}, \cite[Theorem 3.6, (6),(8)]{SVV})} If $\lambda_0\geq 0$ then for any $z\in\RR^{\ell\Lam_0}_{n\alpha_0}$ there are unique elements
$p_{k}(z)\in \RR^{\ell\Lam_0}_{(n-1)\alpha_0}$ and
$\pi(z)\in \RR^{\ell\Lam_0}_{(n-1)\alpha_0}\otimes_{R^{\ell\Lam_0}_{(n-2)\alpha_0}}\RR^{\ell\Lam_0}_{(n-1)\alpha_0}$ such that
$$
z=\mu_{\psi_{n-1}}(\pi(z))+\sum^{\lambda_0-1}_{k=0}p_{k}(z)y^k_{n},
$$
where the above summation is understood as $0$ when $\lam_0=0$.

If $\lambda_0\leq 0$ then for any $z\in \RR^{\ell\Lam_0}_{n\alpha_0}$, there is a unique element
$
\widetilde{z}\in \RR^{\ell\Lam_0}_{(n-1)\alpha_0}\otimes_{\RR^{\ell\Lam_0}_{(n-2)\alpha_0}} \RR^{\ell\Lam_0}_{(n-1)\alpha_0}$ such that
$$
\mu_{\psi_{n-1}}(\widetilde{z})=z,\,\,\, \text{and}\,\,\, \ \mu_{y^{k}_{n-1}}(\widetilde{z})=0,\,\,\, \forall\ k\in \{0, 1,\cdots,-\lambda_0-1\},
$$
where the range of $k$ is understood as $\emptyset$ when $\lam_0=0$.
\end{dfn}

\begin{dfn}\label{SVVdfn1} \text{(\cite[Theorem 3.8]{SVV})} For each $n\in\N$, we define $\hat{\varepsilon}_n: \RR^{\ell\Lam_0}_{n\alpha_0}\rightarrow\RR^{\ell\Lam_0}_{(n-1)\alpha_0}$ as follows:
for any $z\in \RR^{\ell\Lam_0}_{n\alpha_0}$, if $\lambda_0:=\ell-2(n-1)>0$ then $\hat{\varepsilon}_n(z):=p_{\ell-2(n-1)-1}(z)$; if $\lambda_0:=\ell-2(n-1)\leq 0$ then $\hat{\varepsilon}_n(z):=\mu_{y_{n-1}^{-\ell+2(n-1)}}(\widetilde{z})$.
\end{dfn}

\begin{dfn}\label{SVVdfn2} \text{(\cite[A.3.]{SVV})} For any  $z\in \RR^{\ell\Lam_0}_{n\alpha_0}$, $$
\Tr^{\text{SVV}}(z):=\hat{\varepsilon}_1\circ\hat{\varepsilon}_2\circ\cdots\circ\hat{\varepsilon}_n: \RR^{\ell\Lam_0}_{n\alpha_0}\rightarrow  \RR^{\ell\Lam_0}_{0\alpha_0}=K .
$$
\end{dfn}

\begin{dfn} \label{z0n} For each $n\in\N$, we define $$
Z_{0,n}:=\psi_{w_{0,n}}y_1^{\ell-1}y_2^{\ell-2}\cdots y_n^{\ell-n}\in\HH_{\ell,n}^{(0)} .
$$
\end{dfn}

We want to compute the value $\Tr^{\text{SVV}}(Z_{0,n})$. According to Definition \ref{ptilde}, we need to understand the value $p_{\ell-2(n-1)-1}(Z_{0,n})$ when $\ell>2(n-1)$ and the value $\mu_{y_{n-1}^{-\ell+2(n-1)}}(\widetilde{Z_{0,n}})$ when $\ell\leq 2(n-1)$.

\begin{lem} \label{case1} Suppose that $\lambda_0:=\ell-2(n-1)\geq 0$. Then $$\begin{aligned}
\pi(Z_{0,n})=&\bigl((\psi_{1}\cdots \psi_{n-2})y^{\ell-n}_{n-1}\bigr)\otimes (\psi_{1}\cdots \psi_{n-3}\psi_{n-2})\cdots (\psi_{1}\psi_{2})\psi_{1}y^{\ell-1}_1y^{\ell-2}_2\cdots y^{\ell-n+1}_{n-1}\\
&\qquad \in\RR^{\ell\Lam_0}_{(n-1)\alpha_0}\otimes_{\RR^{\ell\Lam_0}_{(n-2)\alpha_0}} \RR^{\ell\Lam_0}_{(n-1)\alpha_0},
\end{aligned}
$$
and for any $k\in \{0,1,\cdots, \lambda_0-1\}$,
$$
p_{k}(Z_{0,n})=(\psi_{1}\cdots \psi_{n-2})(\psi_{1}\cdots \psi_{n-3})\cdots (\psi_{1}\psi_{2})\psi_{1}y^{\ell-1}_1y^{\ell-2}_2\cdots y_{n-2}^{\ell-n+2}y^{\ell-n+\lam_0-k}_{n-1} .
$$
In particular, $p_{\lam_0-1}(Z_{0,n})=Z_{0,n-1}$.
\end{lem}

\begin{proof} By definition, we have that $$ \begin{aligned}
Z_{0,n}&=\psi_{w_{0,n}}y^{\ell-1}_1y^{\ell-2}_2\cdots y^{\ell-n}_n  \\
&=(\psi_{1}\cdots \psi_{n-2}\psi_{n-1})(\psi_{1}\cdots \psi_{n-3}\psi_{n-2})\cdots (\psi_{1}\psi_{2})\psi_{1}y^{\ell-1}_1y^{\ell-2}_2\cdots y^{\ell-n}_n\\
&=(\psi_{1}\cdots \psi_{n-2})(\psi_{n-1}y^{\ell-n}_n)\psi_{w_{0,n-1}}y^{\ell-1}_1y^{\ell-2}_2\cdots y^{\ell-n+1}_{n-1}\\
&=(\psi_{1}\cdots \psi_{n-2})\biggl(y^{\ell-n}_{n-1}\psi_{n-1}+\sum_{\substack{a_1+a_2=\ell-n-1\\ a_1,a_2\geq 0}}y^{a_1}_{n-1}y^{a_2}_{n}\biggr)\psi_{w_{0,n-1}}y^{\ell-1}_1y^{\ell-2}_2\cdots y^{\ell-n+1}_{n-1}\\
&=(\psi_{1}\cdots \psi_{n-2})(y^{\ell-n}_{n-1}\psi_{n-1})\psi_{w_{0,n-1}}y^{\ell-1}_1y^{\ell-2}_2\cdots y^{\ell-n+1}_{n-1}
+\\
&\qquad\qquad\sum_{\substack{a_1+a_2=\ell-n-1\\ a_1,a_2\geq 0}}\biggl(\psi_{1}\cdots \psi_{n-2}y^{a_1}_{n-1}\psi_{w_{0,n-1}}y^{\ell-1}_1y^{\ell-2}_2\cdots y^{\ell-n+1}_{n-1}y^{a_2}_{n}\biggr)\\
&=(\psi_{1}\cdots \psi_{n-2})(y^{\ell-n}_{n-1}\psi_{n-1})\psi_{w_{0,n-1}}y^{\ell-1}_1y^{\ell-2}_2\cdots y^{\ell-n+1}_{n-1}
+\sum_{\substack{a_1+a_2=\ell-n-1\\ a_1,a_2\geq 0}}\biggl(\psi_{1}\cdots \psi_{n-2}y^{a_1}_{n-1}(\psi_1\cdots\psi_{n-3}\psi_{n-2})\\
\\
&\qquad\qquad (\psi_1\cdots\psi_{n-4}\psi_{n-3})
\cdots (\psi_1\psi_2)\psi_1 y^{\ell-1}_1y^{\ell-2}_2\cdots y^{\ell-n+1}_{n-1}y^{a_2}_{n}\biggr)\\
&=(\psi_{1}\cdots \psi_{n-2})(y^{\ell-n}_{n-1}\psi_{n-1})\psi_{w_{0,n-1}}y^{\ell-1}_1y^{\ell-2}_2\cdots y^{\ell-n+1}_{n-1}
+\sum_{\substack{a_1+a_2=\ell-n-1\\ a_1,a_2\geq 0}}\biggl((\psi_{1}\cdots \psi_{n-2})(\psi_1\cdots\psi_{n-3})\\
&\qquad\qquad (\psi_1\cdots\psi_{n-4})\cdots \psi_1 y^{a_1}_{n-1} (\psi_{n-2}\psi_{n-3}\cdots \psi_2\psi_1)y^{\ell-1}_1y^{\ell-2}_2\cdots y^{\ell-n+1}_{n-1}y^{a_2}_{n}\biggr)\\
&=\mu_{\psi_{n-1}}\biggl((\psi_{1}\cdots \psi_{n-2}y^{\ell-n}_{n-1})\otimes (\psi_{w_{0,n-1}}y^{\ell-1}_1y^{\ell-2}_2\cdots y^{\ell-n+1}_{n-1})\biggr)+\\
&\qquad\qquad\sum_{\substack{a_1+a_2=\ell-n-1\\ a_1,a_2\geq 0}}\psi_{w_{0,n-1}}(y^{a_1}_{n-1}\psi_{n-2}\cdots \psi_{2}\psi_{1})y^{\ell-1}_1y^{\ell-2}_2\cdots y^{\ell-n+1}_{n-1}y^{a_2}_{n} .
\end{aligned}
$$

Using the uniqueness in Definition \ref{ptilde}, we see that to prove the lemma, it suffices to show that $$ \begin{aligned}
&\quad\,\sum_{\substack{a_1+a_2=\ell-n-1\\ a_1,a_2\geq 0}}\psi_{w_{0,n-1}}(y^{a_1}_{n-1}\psi_{n-2}\cdots \psi_{2}\psi_{1})y^{\ell-1}_1y^{\ell-2}_2\cdots y^{\ell-n+1}_{n-1}y^{a_2}_{n}\\
&=\sum_{k=0}^{\lam_0-1}\psi_{w_{0,n-1}}y^{\ell-1}_1y^{\ell-2}_2\cdots y^{\ell-n+2}_{n-2}y^{\ell-n+\lam_0-k}_{n-1}y^{k}_{n} .
\end{aligned}
$$

In fact, $$\begin{aligned}
&\quad\,\sum_{\substack{a_1+a_2=\ell-n-1\\ a_1,a_2\geq 0}}\psi_{w_{0,n-1}}(y^{a_1}_{n-1}\psi_{n-2}\cdots \psi_{2}\psi_{1})y^{\ell-1}_1y^{\ell-2}_2\cdots y^{\ell-n+1}_{n-1}y^{a_2}_{n}\\
&=\sum_{\substack{a_1+a_2=\ell-n-1\\ a_1,a_2\geq 0}}\psi_{w_{0,n-1}}(y^{a_1}_{n-1}\psi_{n-2}\cdots \psi_{2}\psi_{1})y^{\ell-1}_1y^{\ell-2}_2\cdots y^{\ell-n+1}_{n-1}y^{a_2}_{n}\\
&=\sum_{\substack{a_1+a_2=\ell-n-1\\ a_1\geq n-2, a_2\geq 0}}\psi_{w_{0,n-1}}(y^{a_1}_{n-1}\psi_{n-2}\cdots \psi_{2}\psi_{1})y^{\ell-1}_1y^{\ell-2}_2\cdots y^{\ell-n+1}_{n-1}y^{a_2}_{n}\\
&=\psi_{w_{0,n-1}}\psi_{n-2}\cdots \psi_{2}\psi_{1}y^{\ell-1}_1y^{\ell-2}_2\cdots y^{\ell-n+2}_{n-2} y^{\ell-n+1}_{n-1}y^{\ell-2n+1}_{n}+\psi_{w_{0,n-1}}y^{\ell-1}_1y^{\ell-2}_2\cdots y^{\ell-n+2}_{n-2}y^{\ell-n+2}_{n-1}y^{\ell-2n}_{n}\\
&\quad\,+\psi_{w_{0,n-1}}y^{\ell-1}_1y^{\ell-2}_2\cdots y^{\ell-n+2}_{n-2}y^{\ell-n+3}_{n-1}y^{\ell-2n-1}_{n}+\cdots+\psi_{w_{0,n-1}}y^{\ell-1}_1y^{\ell-2}_2\cdots y^{\ell-n+2}_{n-2}y^{2\ell-3n+1}_{n-1}y_{n}\\
&\quad\,+\psi_{w_{0,n-1}}y^{\ell-1}_1y^{\ell-2}_2\cdots y^{\ell-n+2}_{n-2}y^{2\ell-3n+2}_{n-1}\\
&=\sum_{k=0}^{\lam_0-1}\psi_{w_{0,n-1}}y^{\ell-1}_1y^{\ell-2}_2\cdots y^{\ell-n+2}_{n-2}y^{\ell-n+\lam_0-k}_{n-1}y^{k}_{n}  ,
\end{aligned}
$$
where we have used the commutator relations for the $\psi$ and $y$ generators of $\HH_{\ell,n}^{(0)}$ and the fact that $\psi_{w_{0,n-1}}\psi_r=0$ for any $1\leq r<n-1$ in the second and the last equalities. This completes the proof of the lemma.
\end{proof}

\begin{lem} \label{case2} Suppose that $\lambda_0:=\ell-2(n-1)\leq 0$. Then
$$
\begin{aligned}
\widetilde{Z_{0,n}}&=\bigl((\psi_{1}\psi_2\cdots \psi_{n-2})y^{\ell-n}_{n-1}\bigr)\otimes \bigl((\psi_{1}\cdots \psi_{n-3}\psi_{n-2})\cdots (\psi_{1}\psi_{2})\psi_{1}y^{\ell-1}_1y^{\ell-2}_2\cdots y^{\ell-n+1}_{n-1}\bigr)\\
&\qquad \in \RR^{\ell\Lam_0}_{(n-1)\alpha_0}\otimes_{R^{\ell\Lam_0}_{(n-2)\alpha_0}}\RR^{\ell\Lam_0}_{(n-1)\alpha_0} .
\end{aligned}
$$
Furthermore, in this case we have that $$
\mu_{y_{n-1}^{-\lam_0}}(\widetilde{Z_{0,n}})=Z_{0,n-1}=(\psi_{1}\cdots \psi_{n-2})(\psi_{1}\cdots \psi_{n-3})\cdots (\psi_{1}\psi_{2})\psi_{1}y^{\ell-1}_1y^{\ell-2}_2\cdots y^{\ell-n+1}_{n-1}.
$$
\end{lem}

\begin{proof} By definition, we have that $$ \begin{aligned}
Z_{0,n}&=\psi_{w_{0,n}}y^{\ell-1}_1y^{\ell-2}_2\cdots y^{\ell-n}_n \\
&=(\psi_{1}\cdots \psi_{n-2}\psi_{n-1})(\psi_{1}\cdots \psi_{n-3}\psi_{n-2})\cdots (\psi_{1}\psi_{2})\psi_{1}y^{\ell-1}_1y^{\ell-2}_2\cdots y^{\ell-n}_n\\
&=(\psi_{1}\cdots \psi_{n-2})(\psi_{n-1}y^{\ell-n}_n)(\psi_{1}\cdots \psi_{n-3}\psi_{n-2})\cdots (\psi_{1}\psi_{2})\psi_{1}y^{\ell-1}_1y^{\ell-2}_2\cdots y^{\ell-n+1}_{n-1}\\
&=(\psi_{1}\cdots \psi_{n-2})\biggl(y^{\ell-n}_{n-1}\psi_{n-1}+\sum_{\substack{a_1+a_2=\ell-n-1\\ a_1,a_2\geq 0}}y^{a_1}_{n-1}y^{a_2}_{n}\biggr)\psi_{w_{0,n-1}}y^{\ell-1}_1y^{\ell-2}_2\cdots y^{\ell-n+1}_{n-1}\\
&=(\psi_{1}\cdots \psi_{n-2})(y^{\ell-n}_{n-1}\psi_{n-1})\psi_{w_{0,n-1}}y^{\ell-1}_1y^{\ell-2}_2\cdots y^{\ell-n+1}_{n-1}+\\
&\qquad\quad\sum_{\substack{a_1+a_2=\ell-n-1\\ a_1,a_2\geq 0}}\psi_{1}\cdots \psi_{n-2}(y^{a_1}_{n-1})\psi_{w_{0,n-1}}y^{\ell-1}_1y^{\ell-2}_2\cdots y^{\ell-n+1}_{n-1}y^{a_2}_{n} .
\end{aligned}
$$

We now claim that \begin{equation}\label{sum01}
\sum_{\substack{a_1+a_2=\ell-n-1\\ a_1,a_2\geq 0}}\psi_{1}\cdots \psi_{n-2}(y^{a_1}_{n-1})\psi_{w_{0,n-1}}y^{\ell-1}_1y^{\ell-2}_2\cdots y^{\ell-n+1}_{n-1}y^{a_2}_{n}=0 .
\end{equation}

In fact, we have that $$ \begin{aligned}
&\quad\,\sum_{\substack{a_1+a_2=\ell-n-1\\ a_1,a_2\geq 0}}\psi_{1}\cdots \psi_{n-2}(y^{a_1}_{n-1})\psi_{w_{0,n-1}}y^{\ell-1}_1y^{\ell-2}_2\cdots y^{\ell-n+1}_{n-1}y^{a_2}_{n}\\
&=\sum_{\substack{a_1+a_2=\ell-n-1\\ a_1,a_2\geq 0}}\psi_{1}\cdots \psi_{n-2}(y^{a_1}_{n-1})(\psi_1\cdots\psi_{n-3}\psi_{n-2})(\psi_1\cdots\psi_{n-4}\psi_{n-3})
\cdots (\psi_1\psi_2)(\psi_1)y^{\ell-1}_1y^{\ell-2}_2\cdots y^{\ell-n+1}_{n-1}y^{a_2}_{n}\\
&=\sum_{\substack{a_1+a_2=\ell-n-1\\ a_1,a_2\geq 0}}\psi_{w_{0,n-1}}(y^{a_1}_{n-1}\psi_{n-2}\cdots \psi_{2}\psi_{1})y^{\ell-1}_1y^{\ell-2}_2\cdots y^{\ell-n+1}_{n-1}y^{a_2}_{n}\\
&=\sum_{\substack{a_1+a_2=\ell-n-1\\ a_1>0, a_2\geq 0}}\psi_{w_{0,n-1}}(y^{a_1}_{n-1}\psi_{n-2}\cdots \psi_{2}\psi_{1})y^{\ell-1}_1y^{\ell-2}_2\cdots y^{\ell-n+1}_{n-1}y^{a_2}_{n} ,
\end{aligned}
$$
where the last equality follows from the fact that $\psi_{w_{0,n-1}}\psi_{n-2}=0$. Now by assumption, $a_1\leq \ell-n-1\leq 2(n-1)-n-1=n-3<n-2$. It follows that $y^{a_1}_{n-1}\psi_{n-2}\cdots \psi_{2}\psi_{1}$ is a sum of some elements which have a left factor of the form $\psi_r$ for some $1\leq r<n-1$. Therefore, using the fact that $\psi_{w_{0,n-1}}\psi_r=0$ for any $1\leq r<n-1$ again, we can deduce that the above sum is $0$. This completes the proof of the claim (\ref{sum01}).

By Definition \ref{ptilde}, to complete the proof of the lemma, it remains to show that for any $0\leq k\leq-\lam_0-1$, \begin{equation}\label{sum02}
\mu_{y^k_{n-1}}\biggl((\psi_{1}\psi_2\cdots \psi_{n-2}y^{\ell-n}_{n-1})\otimes \bigl(\psi_{w_{0,n-1}}y^{\ell-1}_1y^{\ell-2}_2\cdots y^{\ell-n+1}_{n-1}\bigr)\biggr)=0 .
\end{equation}

In fact, we have that $$
\begin{aligned}
&\quad\,\mu_{y^k_{n-1}}\biggl((\psi_{1}\psi_2\cdots \psi_{n-2}y^{\ell-n}_{n-1})\otimes \bigl(\psi_{w_{0,n-1}}y^{\ell-1}_1y^{\ell-2}_2\cdots y^{\ell-n+1}_{n-1}\bigr)\biggr)\\
&=\mu_{y^k_{n-1}}\biggl((\psi_{1}\cdots \psi_{n-2}y^{\ell-n}_{n-1})\otimes (\psi_{1}\cdots \psi_{n-3}\psi_{n-2})\cdots (\psi_{1}\psi_{2})\psi_{1}y^{\ell-1}_1y^{\ell-2}_2\cdots y^{\ell-n+1}_{n-1}\biggr)\\
&=(\psi_{1}\cdots \psi_{n-2})(y^{\ell-n+k}_{n-1})(\psi_{1}\cdots \psi_{n-3}\psi_{n-2})(\psi_{1}\cdots \psi_{n-4}\psi_{n-3})\cdots (\psi_{1}\psi_{2})\psi_{1}y^{\ell-1}_1y^{\ell-2}_2\cdots y^{\ell-n+1}_{n-1}\\
&=(\psi_{1}\cdots \psi_{n-2})(\psi_{1}\cdots \psi_{n-3})\cdots (\psi_{1}\psi_{2})\psi_{1}(y^{\ell-n+k}_{n-1}\psi_{n-2}\psi_{n-3}\cdots \psi_2\psi_{1})y^{\ell-1}_1y^{\ell-2}_2\cdots y^{\ell-n+1}_{n-1}\\
&=\psi_{w_{0,n-1}}(y^{\ell-n+k}_{n-1}\psi_{n-2}\psi_{n-3}\cdots \psi_{1})y^{\ell-1}_1y^{\ell-2}_2\cdots y^{\ell-n+1}_{n-1}\\
&=0 ,
\end{aligned}
$$
where the last equality follows from the fact that $\psi_{w_{0,n-1}}\psi_r=0$ for any $1\leq r<n-1$ and the assumption that $$
\ell-n+k\leq \ell-n-\lambda_0-1=\ell-n-(\ell-2(n-1))-1=n-3<n-2
$$
so that $y^{\ell-n+k}_{n-1}\psi_{n-2}\psi_{n-3}\cdots \psi_{1}$ is a sum of some elements which have a left factor of the form $\psi_r$ for some $1\leq r<n-1$.
This completes the proof of (\ref{sum02}) and hence the proof of the lemma.
\end{proof}

\begin{cor} \label{SVVtrace} We have that $\Tr^{\text{SVV}}(Z_{0,n})=1$.
\end{cor}

\begin{proof} This follows from Definition \ref{SVVdfn1}, Definition \ref{SVVdfn2}, Lemma \ref{case1}, Lemma \ref{case2} and an induction on $n$.
\end{proof}

\begin{thm} \label{mainthm3} The two symmetrizing forms $\Tr^{\text{SVV}}$ and $\Tr$ on the cyclotomic nilHecke algebra $\HH_{\ell,n}^{(0)}$ coincide with each other.
\end{thm}

\begin{proof} Let $1\leq i,j\leq n!$ and $z\in Z$. Suppose that $i\neq j$. Then as $\Tr^{\text{SVV}}$ is a symmetrizing form and $z$ is central, we have that $$
\begin{aligned}
\Tr^{\text{SVV}}(\wf_{w_i,w_j}z)&=\Tr^{\text{SVV}}(\wf_{w_i,w_i}\wf_{w_i,w_j}z)=\Tr^{\text{SVV}}(\wf_{w_i,w_j}z\wf_{w_i,w_i})=
\Tr^{\text{SVV}}(\wf_{w_i,w_j}\wf_{w_i,w_i}z)\\
&=\Tr^{\text{SVV}}(0z)=0 .
\end{aligned}$$
It remains to consider the case when $i=j$.

If $\deg z<d_\Lam$, then as $\Tr^{\text{SVV}}$ is homogeneous of degree $-d_\Lam$ and $\deg\wf_{w_i,w_i}=0$, we have that $\Tr^{\text{SVV}}(\wf_{w_i,w_i}z)=0$. Therefore, without loss of generality, we can assume that $z=z_{\blam_\maxim}$. Our purpose is to compare $\Tr^{\text{SVV}}(\wf_{w_i,w_i}z_{\blam_\maxim})$ and
$\Tr(\wf_{w_i,w_i}z_{\blam_\maxim})$.

Note that for any $\bmu\in\P_0$ with $\bmu>\blam_{\min}$, we have that $$
\deg(y_\bmu z_{\blam_{\max}})>n(n-1)+2n(\ell-n)=2\ell n-n(n+1)=\deg(y_1^{\ell-1}y_2^{\ell-2}\cdots y_n^{\ell-n}),
$$
which implies that $y_\bmu z_{\blam_{\max}}=0$ by Theorem \ref{mainthm1}. By (\ref{leadingterm2}) and Lemma \ref{keylem2}, we have that $$\begin{aligned}
\Tr^{\text{SVV}}(\wf_{w_i,w_i}z_{\blam_\maxim})&=(-1)^{n(n-1)/2}\Tr^{\text{SVV}}(\psi_{w_0w_i,w_i}^{\blam_\minum}z_{\blam_\maxim})=
\Tr^{\text{SVV}}(\psi_{w_0w_i,w_i}^{\blam_\maxim})\\
&=\Tr^{\text{SVV}}(\psi_{w_i}\psi_{w_0w_i}^{\ast}y_1^{\ell-1}y_2^{\ell-2}\cdots y_n^{\ell-n})\\
&=\Tr^{\text{SVV}}(\psi_{w_0}y_1^{\ell-1}y_2^{\ell-2}\cdots y_n^{\ell-n})\\
&=\Tr^{\text{SVV}}(Z_{0,n})=1, \quad\text{(by Lemma \ref{SVVtrace})}\\
\Tr(\wf_{w_i,w_i}z_{\blam_\maxim})&=(-1)^{n(n-1)/2}\Tr(\psi_{w_0w_i,w_i}^{\blam_\minum}z_{\blam_\maxim})=
\Tr(\psi_{w_0w_i,w_i}^{\blam_\maxim})\\
&=\Tr(\psi_{w_i}\psi_{w_0w_i}^{\ast}y_1^{\ell-1}y_2^{\ell-2}\cdots y_n^{\ell-n})\\
&=\Tr(\psi_{w_0}y_1^{\ell-1}y_2^{\ell-2}\cdots y_n^{\ell-n})\\
&=1 .  \quad\text{(by (\ref{maxvalue}))}
\end{aligned}
$$
This shows that $\Tr^{\text{SVV}}(\wf_{w_i,w_i}z_{\blam_\maxim})=\Tr(\wf_{w_i,w_i}z_{\blam_\maxim})$.

As a result, we have shown that $\Tr^{\text{SVV}}(\wf_{w_i,w_j}z)=\Tr(\wf_{w_i,w_j}z)$ for any $1\leq i,j\leq n!$ and $z\in Z$. It follows that
$\Tr^{\text{SVV}}=\Tr$ as required.
\end{proof}


\bigskip
\begin{thebibliography}{2}

\bibitem{AK1}
{\sc S.~Ariki and K.~Koike}, {\em A Hecke algebra of $(\Z/r\Z)\wr S_n$ and construction of its irreducible
representations}, Adv. Math., {\bf 106} (1994), 216--243.

\bibitem{AM} {\sc S. Ariki and A. Mathas}, {\em The number of simple
modules of the Hecke algebras of type $G(r,1,n)$}, Math. Z., {\bf 233}(3) (2000), 601--623.

\bibitem{Brundan:degenCentre}
{\sc J.~Brundan}, {{\em  Centers of degenerate cyclotomic {H}ecke algebras and parabolic category
  {$\mathcal O$}}}, Represent. Theory, {\bf 12} (2008), 236--259.

\bibitem{BK:HigherSchurWeyl}
{\sc J.~Brundan and A.~Kleshchev},
  {{\em Schur-{W}eyl duality for higher levels}}, Selecta Math. (N.S.), {\bf 14} (2008), 1--57.

\bibitem{BK:GradedKL}
\leavevmode\vrule height 2pt depth -1.6pt width 23pt,
{{\em Blocks of cyclotomic
  {H}ecke algebras and {K}hovanov-{L}auda algebras}}, Invent. Math., {\bf 178}
  (2009), 451--484.

\bibitem{BK:GradedDecomp}
\leavevmode\vrule height 2pt depth -1.6pt width 23pt,
  {{\em Graded decomposition numbers for cyclotomic {H}ecke algebras}}, Adv. Math., {\bf 222} (2009),
  1883--1942.

\bibitem{BKW:GradedSpecht}
{\sc J.~Brundan, A.~Kleshchev and W.~Wang},
  {{\em Graded {S}pecht modules}}, J. Reine Angew. Math., {\bf 655} (2011), 61--87.




\bibitem{CuR} {\sc C.~W. Curtis and L. Reiner}, {\em Methods of Representations
  Theory, I}, Wiley-Interscience, New York, 1981.

\bibitem{DJ1} {\sc R.~Dipper and G.~D. James}, {\em Representations of
Hecke algebras of general linear groups}, Proc. London. Math. Soc.,
{\bf 52}(3) (1986), 20--52.


\bibitem{DJM:cyc}
{\sc R.~Dipper, G.~James and A.~Mathas}, {{\em Cyclotomic {$q$}-{S}chur algebras}}, Math. Z., {\bf 229} (1998), 385--416.

\bibitem{GL}
{\sc J.~J. Graham and G.~I. Lehrer},
  {{\em Cellular algebras}}, Invent. Math., {\bf 123} (1996), 1--34.

\bibitem{GP}
{\sc M.~Geck and G.~Pfeiffer}, {\em Characters of finite Coxeter groups and Iwahori-Hecke algebras}, Clarendon Press, Oxford, 2000.


\bibitem{Hi}
{\sc H.~Hiller}, {\em Geometry of Coxeter groups}, Res. Notes Math., vol {\bf 54}, Pitman (Advanced Publishing Program), Boston, MA, 1982.

\bibitem{HoffnungLauda:KLRnilpotency}
{\sc A.~E. Hoffnung and A.~D. Lauda},
  {{\em Nilpotency in type {$A$} cyclotomic quotients}}, J. Algebraic Combin., {\bf 32} (2010),
  533--555.

\bibitem{Hlam}
{\sc J. Hu and N. Lam}, {\em Symmetric structure for the endomorphism algebra of projective-injective module in parabolic Category}, preprint,
arXiv:1702.05834, 2017.

\bibitem{HuMathas:GradedCellular} {\sc J.~Hu and A.~Mathas},  {{\em Graded cellular bases
  for the cyclotomic {K}hovanov-{L}auda-{R}ouquier algebras of type {$A$}}},   Adv. Math., {\bf 225} (2010), 598--642.

\bibitem{HuMathas:GradedInduction}
\leavevmode\vrule height 2pt depth -1.6pt width 23pt, {{\em
  Graded induction for {S}pecht modules}}, Int. Math. Res. Not., {\bf
  2012} (2012), 1230--1263.

\bibitem{K}
{\sc M.~Kashiwara}, {\em Biadjointness in cyclic Khovanov-Lauda-Rouquier algebras}, Publ. Res. Inst. Math. Sci., {\bf 48}(3) (2012), 501--524.

\bibitem{KK}
{\sc S.~J. Kang and M.~Kashiwara}, {\em Categorification of highest weight modules via Khovanov-Lauda-Rouquier algebras}, Invent. Math., {\bf 190} (2012), 699--742.

\bibitem{KhovLaud:diagI}
{\sc M.~Khovanov and A.~D. Lauda}, {{\em A diagrammatic
  approach to categorification of quantum groups. {I}}}, Represent. Theory,
  {\bf 13} (2009), 309--347.


\bibitem{KLMS}
{\sc M.~Khovanov, A.~Lauda, M.~Mackaay and M.~Stosic}, {\em Extended graphical calculus for categorified quantum $sl(2)$}, Memoirs of the AMS, {\bf 219} (2012).

\bibitem{KLM}
{\sc A. Kleshchev, J. Loubert and V. Miemietz}, {\em Affine Cellularity of Khovanov-Lauda-Rouquier algebras in type $A$}, J. London Math. Soc.,
   {\bf 88} (2013), 338--358.


\bibitem{KoK} {\sc B. Kostant and S, Kumar}, {\em The nilHecke ring and cohomology of $G/P$ for a Kac--Moody groups}, Adv. Math., {\bf 62}(3) (1986), 187--237.

\bibitem{Lau}
{\sc A.~Lauda}, {\em A categorification of quantum $sl(2)$}, Adv. Math., {\bf 225} (2010), 3327--3424.

\bibitem{Lau2}
{\sc A.~Lauda}, {\em An introduction to diagrammatic algebra and categorified quantum $sl(2)$}, Bulletin of the Institute of Mathematics Academia Sinica, {\bf 7} (2012), No. 2, 165--270.




\bibitem{LV}
{\sc A.~Lauda and M.~Vazirani}, {\em Crystals from categorified quantum groups}, Adv. Math., {\bf 228} (2011), 803--861.




\bibitem{Man}
{\sc L.~Manivel}, {\em Symmetric functions, Schubert polynomials and degeneracy loci}, SMF/AMS Texts and
Monographs 6 (American Mathematical Society, Providence, RI, 2001).


\bibitem{Mathas:Singapore}
{\sc A.~Mathas}, {{\em Cyclotomic quiver
  {H}ecke algebras of type {A}}}, in Modular representation theory of finite
  and $p$-adic groups, G.~W. Teck and K.~M. Tan, eds., National University of
  Singapore Lecture Notes Series, {\bf 30}, World Scientific, 2015, ch.~5,
  165--266.

\bibitem{MM} {\sc G.~Malle and A.~Mathas}, {\em Symmetric cyclotomic Hecke algebras},
  J. Alg., {\bf 205} (1998), 275--293.

\bibitem{Rou0}
{\sc R.~Rouquier}, {\em Quiver Hecke algebras and 2-Lie algebras}, Algebra Colloq. {\bf 19} (2012), 359--410.

\bibitem{Rou1}
\leavevmode\vrule height 2pt depth -1.6pt width 23pt, {\em 2-Kac-Moody algebras}, preprint, ArXiv:0812.5023, 2008.

\bibitem{SVV}
{\sc P.~Shan, M.~Varagnolo and E.~Vasserot}, {\em On the center of quiver-Hecke algebras}, Duke Math. J., {\bf 166}(6) (2017), 1005--1101.

\bibitem{VaragnoloVasserot:CatAffineKLR}
{\sc M.~Varagnolo and E.~Vasserot}, {\em Canonical bases and {KLR}-algebras}, J. Reine Angew. Math., {\bf 659} (2011), 67--100.

\bibitem{Web}
{\sc B.~Webster}, {\em Center of KLR algebras and cohomology rings of quiver varieties}, preprint, math.RT/1504.04401v2, 2015.

\end{thebibliography}
\end{document}